\documentclass{amsart}
\usepackage{amssymb,latexsym,amsfonts}
\usepackage{amsthm}
\usepackage{fontenc}
\usepackage{amsmath}
\usepackage{geometry}
\usepackage{graphicx}
\usepackage{stmaryrd}
\usepackage[all]{xy}
\usepackage[francais]{babel}
\usepackage[latin1]{inputenc}
\usepackage[T1]{fontenc}
\DeclareMathOperator{\Nil}{Nil}

\DeclareMathOperator{\Sym}{Sym}

\DeclareMathOperator{\Red}{Red}

\DeclareMathOperator{\Spec}{Spec}

\DeclareMathOperator{\supp}{supp}
\DeclareMathOperator{\Id}{Id}

\DeclareMathOperator{\End}{End}
\DeclareMathOperator{\dime}{dim}

\DeclareMathOperator{\Gr}{Gr}

\DeclareMathOperator{\Aut}{Aut}
\DeclareMathOperator{\ad}{ad}
\DeclareMathOperator{\Ad}{Ad}

\DeclareMathOperator{\M}{M}
\DeclareMathOperator{\defa}{def}

\DeclareMathOperator{\Ker}{Ker}

\DeclareMathOperator{\Tr}{Tr}
\DeclareMathOperator{\car}{car}
\DeclareMathOperator{\rg}{rg}

\DeclareMathOperator{\val}{val}

\DeclareMathOperator{\Lie}{Lie}

\newtheorem{thme}{Theorem}
\newtheorem{thm}{Théorème}[section]
\newtheorem{prop}[thm]{Proposition}
\newtheorem{lem}[thm]{Lemme}
\newtheorem{defi}[thm]{Définition}

\newtheorem{cor}[thm]{Corollaire}

\newcommand{\rmq}{\noindent\textbf{Remarque :}}
\newcommand{\rmqs}{\noindent\textbf{Remarques :}}
\newcommand{\cH}{\mathcal{H}}

\newcommand{\cE}{\mathcal{E}}

\newcommand{\GL}{GL}
\newcommand{\bP}{\mathbb{P}}
\newcommand{\co}{\mathcal{O}}

\newcommand{\cB}{\mathcal{B}}
\newcommand{\cL}{\mathcal{L}}
\newcommand{\cm}{\mathcal{M}}

\newcommand{\g}{\gamma}

\newcommand{\la}{\lambda}

\newcommand{\eps}{\epsilon}

\newcommand{\kX}{\mathfrak{X}}

\newcommand{\kc}{\mathfrak{C}_{+}}

\newcommand{\kcd}{\mathfrak{C}_{+}^{\lambda}}
\newcommand{\kg}{\mathfrak{g}}

\newcommand{\kx}{\mathfrak{X}}

\newcommand{\kt}{\mathfrak{t}}

\newcommand{\ev}{ev}

\newcommand{\ab}{\mathbb{A}}

\begin{document}
\title{Dimension des fibres de Springer affines pour les groupes}
\author{Alexis Bouthier}
\maketitle
\tableofcontents

\begin{center}
\textbf{Abstract:}
\end{center}
This article establishes a dimension formula for a group version of affine Springer fibers. We follow the method initiated by Bezrukavnikov in the case of Lie algebras. It consists in the introduction of a big enough regular open subset, with the same dimension as the affine Springer fiber. We show that, in the case of groups, such a regular open subset with analogous properties exists. Its construction needs the introduction of the Vinberg semi-group $V_{G}$ for which we study an adjoint quotient $\chi_{+}$ and extend for $\chi_{+}$ the results previously established by Steinberg.
\bigskip
\bigskip
\begin{center}
\textbf{Résumé:}
\end{center}
Cet article établit une formule de dimension pour les fibres de Springer affines dans le cas des groupes. On suit la méthode initiée par Bezrukavnikov dans le cas des algèbres de Lie. Elle consiste en l'introduction d'un ouvert régulier suffisament gros dont on montre qu'il est de même dimension que la fibre de Springer affine entière.
On montre que dans le cas des groupes, un tel ouvert régulier avec des propriétés analogues, existe. Sa construction passe par l'introduction du semi-groupe de Vinberg $V_{G}$ pour lequel nous étudions un morphisme `polynôme caractéristique' et étendons les résultats précedemment établis par Steinberg pour les groupes.
\bigskip
\section*{Introduction in English}
Let $k$ be an algebraically closed field.
We consider $G$ a connected algebraic group, semisimple, simply connected over $k$. Let $T$ be a maximal torus of $G$ and $\kg$ the Lie algebra of $G$.
We note $F=k((\pi))$ and $\co:=k[[\pi]]$. Kazhdan and Lusztig have introduced in \cite{KL} the affine Springer fibers for Lie algebras. There are varieties of the form
\begin{center}
$\mathfrak{X}_{\gamma}=\{g\in G(F)/G(\mathcal{O})\vert~\ad(g)^{-1}\gamma \in\mathfrak{g}(\mathcal{O})\}$
\end{center}
where $\gamma\in \mathfrak{g}(F)$. They establish that they are $k$-schemes locally of finite type and of finite dimension if $\g$ is regular semisimple. They also conjecture a dimension formula for these varieties which was later proved by Bezrukavnikov \cite{B}. If we name by $\kg_{\g}$, the centralizer of $\g$ in $\kg$, the formula is the following:
\begin{center}
$\dime \mathfrak{X}_{\gamma}=\frac{1}{2}[\delta'(\gamma)-\defa(\gamma)]$
\end{center}
where $\delta'(\gamma)=\val(\det(\ad(\gamma):\mathfrak{g}(F)/\mathfrak{g}_{\gamma}(F)\rightarrow \mathfrak{g}(F)/\mathfrak{g}_{\gamma}(F)))$ and $\defa(\gamma)=\rg\mathfrak{g}-\rg_{F}(\mathfrak{g}_{\gamma}(F))$. The first term is the discriminant invariant and the second one is a Galois invariant which mesures the drop of torus rank.
In this work, we are interested in the affine Springer fibers for groups:
\begin{center}
$X_{\gamma}^{\la}=\{g\in G(F)/G(\mathcal{O})\vert~g^{-1}\gamma g\in G(\mathcal{O})\pi^{\lambda}G(\mathcal{O})\}$
\end{center}
with $\lambda\in X_{*}(T)^{+}$ a dominant cocharacter and $\g\in G(F)$.
These varieties were introduced by Kottwitz-Viehmann \cite{KV} in their article on generalized Springer fibers. If $\la=0$, we have the variety
\begin{center}
$X_{\gamma}^{0}=\{g\in G(F)/G(\mathcal{O})\vert~g^{-1}\gamma g\in G(\mathcal{O})\}$
\end{center}
and the dimension formula and the proof of Bezrukavnikov are the same. For a general $\la$, we prove:
\smallskip
\begin{thme}\label{1b}
Let $\gamma\in G(F)$ be a regular semisimple element. Then:
\begin{enumerate}
\item $X_{\gamma}^{\lambda}$  is a $k$-scheme locally of finite type.
\item Si $X_{\gamma}^{\lambda}$ non vide, $\dime X_{\gamma}^{\lambda}=\left\langle \rho,\lambda\right\rangle+ \frac{1}{2}[\delta(\gamma)-\defa(\gamma)]$,

where $\delta(\gamma)=\val(\det(\Id-\ad(\gamma):\mathfrak{g}(F)/\mathfrak{g}_{\gamma}(F)\rightarrow \mathfrak{g}(F)/\mathfrak{g}_{\gamma}(F)))$.
\end{enumerate}
\end{thme}
An analog dimension formula for affine Deligne-Lusztig was already established by \cite{GHKR} and \cite{Vie}.
\medskip
The proof of Bezrukavnikov uses crucially a distinguished open subset, called the regular open subset. It consists in the elements $g\in\kX_{\g}$ such that $\ad(g)^{-1}(\g)$ is regular when we reduce mod $\pi$. In the case of groups, there is a double difficulty coming from the fact that the condition $g^{-1}\g g\in G(\co)\pi^{\la}G(\co)$ is non-linear and that we cannot give a sense to the reduction modulo $\pi$. To linearize the problem, one way to proceed, is to consider a faithful representation $\rho: G\rightarrow\End(V)$ such that $\rho(G(\co)\pi^{\la}G(\co))\subset\pi^{-N}\End(V)\otimes_{k}\co$, for $N\in\mathbb{N}$. We note that, if we add a central factor to $G$, which acts by multiplication by $\pi$ in $\End(V)\otimes_{k}\co$, we  obtain a similar integrality condition, as in the Lie algebra case. To do that in a uniform way for all groups, there exists a natural envelop, called the Vinberg's semi-group $V_{G}$, introduced by Vinberg \cite{Vi} in characteristic zero and Rittatore \cite{Ri} in arbitrary characteristics. This formulation allows us to define in this context a regular open subset.
 
In the case of Lie algebras, the regular open subset is a torsor under the affine grassmannian of the regular centralizer of $\g$. Following Ngô \cite{N}, the existence of such a regular centralizer comes from the existence of a commutative group scheme $J$, smooth on the adjoint quotient $\kt/W$ and from the existence of a map $\chi^{*}J\rightarrow I$, where $\chi$ is the Chevalley morphism $\chi:\kg\rightarrow\kt/W$ and $I$ the scheme of centralizers over $\kg$.
Moreover, the morphism $\chi^{*}J\rightarrow I$ is an isomorphism over $\kg^{reg}$. One way to obtain the group scheme $J$ is to construct a section to $\chi$, called the Kostant section, and to pullback $I$ by this section.

In our case, we try to obtain a Chevalley type morphism for the Vinberg's semigroup $V_{G}$. It is an algebraic monoïd, i.e. a semigroup with unity, with unit group $G_{+}:=(T\times G)/Z_{G}$ which is open dense. In particular, $V_{G}$ is a partial compactification of $G_{+}$, affine and which contains the toric variety $V_{T}$, the closure of $T_{+}:=(T\times T)/Z_{G}$ in $V_{G}$.

By Steinberg \cite{S}, we have a morphism
\begin{center}
$\chi_{+}: G_{+}\rightarrow T_{+}/W$
\end{center}
and a section to this morphism (the simply connectedness assumption is necessary in order to get a section).
We obtain the following theorem:
\smallskip
\begin{thme}\label{3b}
The Steinberg's morphism extends to a map,
\begin{center}
$\chi_{+}:V_{G}\rightarrow V_{T}/W$,
\end{center}
invariant by conjugation by $G_{+}$.

The morphism $\chi_{+}$ admits a section $\epsilon_{+}:V_{T}/W\rightarrow V_{G}^{reg}$, in the regular locus.
\end{thme}

The existence of this section allows us to construct a regular centralizer $J$ by pulling-back the scheme of centralizers $I$ by $\eps_{+}$. To obtain that $J$ is commutative and smooth, we need more properties of the morphism $\chi_{+}^{reg}$ and in particular its smoothness.
\smallskip
\begin{thme}\label{4b}
The morphism $\chi_{+}^{reg}:V_{G}^{reg}\rightarrow V_{T}/W$ is smooth and its geometric fibers are $G$-orbits.

There exists a unique commutative group scheme $J$, smooth over $V_{T}/W$ with a map $\chi_{+}^{*}J\rightarrow I$, which is an isomorphism over $V_{G}^{reg}$.
\end{thme}

Let us now consider the organization of the paper. It splits in two parts, the first one concerns the proof of the theorems \ref{3b} and \ref{4b}, which are results of group theory and the second part deals with the computation of the dimension of Springer fibers.

In the first section, we prove the theorem \ref{3b}. We introduce the Vinberg's semigroup and the quotient by adjoint action $\chi_{+}$. Over the group of units $G_{+}$ of $V_{G}$, we have a section, or more exactly a family of sections constructed by Steinberg, for which we show that they extend to the Vinberg's semigroup. This allows us to construct the regular centralizer $J$.

In the second section, we obtain the properties of the theorem \ref{4b} on the morphism $\chi_{+}$ and the centralizer $J$. By using the action of the central torus $Z_{+}$ of $G_{+}$, we can reduce the study over the point zero, which is the nilpotent cone. The properties established for the most singular fiber then spread to the other fibers.

In the third section, we introduce the affine Springer fibers for groups and we make the link with the Vinberg's semigroup via the modular interpretation. These Springer fibers admit a distinguished open locus, named regular, which is an orbit under a Picards stack, coming from the regular centralizer.

Finally, in the last section, we show the theorem \ref{1b}. Following Kazhdan-Lusztig, we need to study the equidimensionnality of a corresponding flag variety.
It implies to study more deeply the nilpotents and quasi-unipotents elements of the Vinberg's semigroup. Once this result is obtained, it is sufficient to deduce the dimension of the regular open subset which have the same dimension of the whole Springer fiber and conclude about the dimension formula.

\section*{Introduction}
Soit $k$ un corps algébriquement clos.
On considère $G$ un groupe algébrique, connexe, semi-simple, simplement connexe sur $k$.
Soit $T$ un tore maximal de $G$ et $\mathfrak{g}$ l'algèbre de Lie de $G$.
On pose $F:=k((\pi))$ et $\mathcal{O}:=k[[\pi]]$.
\medskip
Kazhdan et Lusztig ont introduit dans \cite{KL} les fibres de Springer affines pour les algèbres de Lie. Ce sont les variétés de la forme
\begin{center}
$\mathfrak{X}_{\gamma}=\{g\in G(F)/G(\mathcal{O})\vert~\ad(g)^{-1}\gamma \in\mathfrak{g}(\mathcal{O})\}$
\end{center}
où $\gamma\in \mathfrak{g}(F)$.
Ils établissent que ce sont des $k$-schémas localement de type fini et de dimension finie si $\gamma$ est régulier semi-simple. Ils conjecturent également pour ces variétés une formule de dimension qui sera démontrée par Bezrukavnikov \cite{B}. Si l'on désigne par $\mathfrak{g}_{\gamma}$, le centralisateur de $\gamma$ dans $\mathfrak{g}$, la formule est la suivante:
\begin{center}
$\dime \mathfrak{X}_{\gamma}=\frac{1}{2}[\delta'(\gamma)-\defa(\gamma)]$
\end{center}
où $\delta'(\gamma)=\val(\det(\ad(\gamma):\mathfrak{g}(F)/\mathfrak{g}_{\gamma}(F)\rightarrow \mathfrak{g}(F)/\mathfrak{g}_{\gamma}(F)))$ et $\defa(\gamma)=\rg\mathfrak{g}-\rg_{F}(\mathfrak{g}_{\gamma}(F))$. Le premier est l'invariant discriminant et le second l'invariant galoisien qui mesure la chute du rang torique.

Dans ce travail, on s'intéresse aux fibres de Springer affines pour les groupes:
\begin{center}
$X_{\gamma}^{\la}=\{g\in G(F)/G(\mathcal{O})\vert~g^{-1}\gamma g\in G(\mathcal{O})\pi^{\lambda}G(\mathcal{O})\}$
\end{center}
avec $\lambda\in X_{*}(T)^{+}$ un cocaractère dominant et $\g\in G(F)$. Ces variétés ont été introduites par Kottwitz-Viehmann \cite{KV} dans leur article sur les fibres de Springer généralisées. 
Dans le cas $\lambda=0$, nous avons la variété:
\begin{center}
$X_{\gamma}^{0}=\{g\in G(F)/G(\mathcal{O})\vert~g^{-1}\gamma g\in G(\mathcal{O})\}$
\end{center}
et la formule de dimension ainsi que la preuve de Bezrukavnikov sont les mêmes.
Pour $\lambda$ général, on démontre:
\medskip
\begin{thm}\label{1}
Soit $\gamma\in G(F)$ régulier semi-simple. Alors:
\begin{enumerate}
\item $X_{\gamma}^{\lambda}$ est un schéma localement de type fini.
\item Si $X_{\gamma}^{\lambda}$ non vide, $\dime X_{\gamma}^{\lambda}=\left\langle \rho,\lambda\right\rangle+ \frac{1}{2}[\delta(\gamma)-\defa(\gamma)]$,

où $\delta(\gamma)=\val(\det(\Id-\ad(\gamma):\mathfrak{g}(F)/\mathfrak{g}_{\gamma}(F)\rightarrow \mathfrak{g}(F)/\mathfrak{g}_{\gamma}(F)))$.
\end{enumerate}
\end{thm}
Une formule de dimension analogue pour les variétés de Deligne-Lusztig affines a déjà été établie par \cite{GHKR} et \cite{Vie}.
\smallskip

La preuve de Bezrukavnikov  utilise de manière cruciale un ouvert distingué appelé l'ouvert régulier. Il consiste en les éléments $g\in \kx_{\g}$ tels que $\ad(g)^{-1}(\g)$ soit régulier en réduction modulo $\pi$.
Dans le cas des groupes, nous nous trouvons confrontés à la double difficulté que la condition $g^{-1}\g g\in G(\co)\pi^{\la}G(\co)$ est non linéaire et que nous ne pouvons donner un sens à la réduction modulo $\pi$. Pour linéariser le problème, une manière possible de procéder, est de considérer une réprésentation fidèle $\rho: G\rightarrow\End(V)$ de telle sorte que $\rho(G(\co)\pi^{\la}G(\co))\subset\pi^{-N}\End (V)\otimes_{k}\co$, pour $N\in\mathbb{N}$.
On remarque alors que, quitte à ajouter un facteur central à $G$, qui agit par multiplication par $\pi$ dans $\End(V)$, on peut obtenir la condition d'intégralité souhaitée. 
Pour faire cela de manière uniforme pour tous les groupes, il existe une enveloppe naturelle, appelée le semi-groupe de Vinberg $V_{G}$, introduit par Vinberg \cite{Vi} en caractéristique nulle et étendu par Rittatore \cite{Ri} en caractéristique arbitraire.
Cette reformulation nous permet donc de pouvoir définir dans ce contexte un ouvert régulier.

Dans le cas des algèbres de Lie, l'ouvert régulier est un torseur sous la grassmannienne affine du centralisateur régulier de $\g$.
D'après Ngô \cite{N}, l'existence de ce centralisateur régulier vient de l'existence d'un schéma en groupes commutatifs $J$ lisse sur le quotient adjoint pour l'algèbre de Lie $\mathfrak{t}/W$ ainsi que de l'existence d'une flèche $\chi^{*}J\rightarrow I$, où $\chi$ est le morphisme de Chevalley $\mathfrak{g}\rightarrow\mathfrak{t}/W$ et $I$ le schéma des centralisateurs au-dessus de $\mathfrak{g}$. En outre, la flèche $\chi^{*}J\rightarrow I$ est un isomorphisme au-dessus de $\mathfrak{g}^{reg}$. Une des façons d'obtenir un tel schéma $J$ est de construire une section, dite de Kostant, au morphisme de Chevalley et de tirer $I$ par cette section.

Dans le cas qui nous concerne, on cherche donc à obtenir un morphisme de type Chevalley pour le semi-groupe de Vinberg $V_{G}$.
C'est un monoïde algébrique, i.e. un semigroupe avec unité, dont le groupe des inversibles  $G_{+}:=(T\times G)/Z_{G}$  est un ouvert dense. En particulier, le semi-groupe de Vinberg est une compactification partielle de $G_{+}$, affine et qui contient également la variété torique $V_{T}$, adhérence du tore maximal $T_{+}:=(T\times T)/Z_{G}$ de $G_{+}$.

On dispose grâce à Steinberg \cite{S} d'une flèche 
\begin{center}
$\chi_{+}:G_{+}\rightarrow T_{+}/W$
\end{center}
ainsi que d'une section à cette flèche (l'hypothèse de simple connexité étant ici indispensable pour l'existence d'une section).

\medskip
\begin{thm}\label{3}
Le morphisme de Steinberg se prolonge en une flèche,
\begin{center}
$\chi_{+}:V_{G}\rightarrow V_{T}/W$,
\end{center}
invariant par conjugaison par $G_{+}$.

Le morphisme $\chi_{+}$ admet une section $\epsilon_{+}:V_{T}/W\rightarrow V_{G}^{reg}$, à valeurs dans  l'ouvert régulier.
\end{thm}

L'existence de cette section nous permet donc de construire le centralisateur régulier $J$ en tirant le schéma des centralisateurs $I$ par $\eps_{+}$. Pour obtenir que $J$ est un schéma en groupes lisse et commutatif, nous avons besoin d'établir des propriétés sur le morphisme $\chi_{+}^{reg}$ et en particulier sa lissité. Cela fait l'objet de l'énoncé suivant:
\medskip
\begin{thm}\label{4}
Le morphisme $\chi_{+}^{reg}:V_{G}^{reg}\rightarrow V_{T}/W$ est lisse et ses fibres géométriques sont des $G$-orbites.

Enfin, il existe un unique schéma en groupes commutatif $J$, lisse sur $V_{T}/W$, muni d'une flèche $\chi_{+}^{*}J\rightarrow I$ qui est un isomorphisme au-dessus de $V_{G}^{reg}$.
\end{thm}
\medskip

Nous passons maintenant en revue l'organisation de l'article. Il se décompose en deux parties, la première concerne la preuve des théorèmes \ref{3} et \ref{4},
qui sont des résultats de théorie des groupes et la deuxième partie concerne le calcul de dimension des fibres de Springer.
\medskip

Dans la première section, on démontre le théorème \ref{3}. Nous introduisons le semi-groupe de Vinberg ainsi que le quotient par l'action adjointe $\chi_{+}$. Au-dessus de l'ouvert des inversibles $G_{+}$ de $V_{G}$, on dispose d'une section, ou plus exactement d'une famille, construite par Steinberg, dont on montre qu'elle se prolonge au semi-groupe de Vinberg. Cela nous permet alors de pouvoir construire le centralisateur régulier $J$.
\medskip

Dans la deuxième section, on obtient les propriétés énoncées dans le théorème \ref{4} sur le morphisme $\chi_{+}$ et le centralisateur $J$. En utilisant l'action du tore central $Z_{+}$ de $G_{+}$, on se ramène à l'étude au-dessus du point zéro, qui s'identifie au cône nilpotent. Les propriétés que l'on démontre pour la fibre la plus `singulière' de $\chi_{+}$ se propagent ensuite aux autre fibres.
\medskip

Dans la troisième section, on introduit les fibres de Springer affines pour les groupes et on fait le lien avec le semi-groupe de Vinberg par l'intermédiaire de l'interprétation modulaire. Ces fibres de Springer admettent un ouvert distingué, dit régulier, qui est une orbite sous un champ de Picard, issu du centralisateur régulier.
\medskip

Enfin, dans la dernière section, on démontre le théorème \ref{1}. Suivant Kazhdan-Lusztig, nous avons besoin d'étudier l'équidimensionnalité d'une variété de drapeaux associée. Cela nécessite une étude approfondie des éléments nilpotents et quasi-unipotents du semi-groupe de Vinberg (sect.\ref{casgen}). Une fois ce résultat d'équidimensionnalité obtenu, cela suffit pour en déduire que la dimension de l'ouvert régulier est la même que celle de la fibre toute entière et conclure quant à la dimension des fibres de Springer affines en général.
\bigskip

Ce travail a fait l'objet d'innombrables navettes entre Gérard Laumon et Ngô Bao Châu qui, par leur relecture avisée et leur soutien ont contribué à améliorer de manière significative la qualité de ce travail, qu'ils en soient ici profondément remerciés. Je tiens également à exprimer ma reconnaissance envers Michel Brion pour ses utiles remarques concernant les subtilités de la caractéristique $p$.

Je remercie l'Université de Chicago pour m'avoir invité par deux fois ainsi que Zongbin Chen pour d'utiles remarques sur les fibres de Springer affines. Enfin, merci à Ivan Boyer et Paul Mercat pour la peine qu'ils ont pris à relire mon piètre LaTeX.

\section{Le semi-groupe de Vinberg et sa section}
\subsection{Rappels sur le semi-groupe de Vinberg}\label{rapsemi}
Soit $k$ un corps algébriquement clos.

Soit $G$  un groupe connexe, semi-simple, simplement connexe sur $k$, de rang $r$.
Soit $(B,T)$ une paire de Borel. On considère l'ensemble des poids fondamentaux $\omega_{1}$,$\dots$, $\omega_{r}$, l'ensemble des racines simples $\Delta=\{\alpha_{1},\dots, \alpha_{r}\}$ et $R$ l'ensemble des racines.

Si $\lambda$ est un cocaractère dominant de $T$, on note $\rho_{\lambda}$, la représentation irréductible de plus haut poids $\lambda$, d'espace vectoriel $V_{\lambda}$.
Enfin, on pose, $\chi_{\lambda}=\Tr(\rho_{\lambda})$.
Tous les résultats énoncés ici, qui concernent les propriétés générales du semi-groupe de Vinberg seront dûs à Vinberg en caractéristique nulle et à Rittatore en caractéristique $p$.

\begin{defi}
Un semi-groupe algébrique $S$ est un $k$-schéma de type fini muni d'un morphisme associatif:
$m:S\times S\rightarrow S$ qui est un morphisme de $k$-schémas.
\end{defi}
Un semi-groupe est dit \textit{irréductible} (resp. normal), si $S$ l'est en tant que $k$-schéma.
Un semi-groupe qui admet une unité pour le morphisme $m$ est appelé un \textit{monoïde}.
Nous parlons de \textit{monoïde algébrique} pour un semi-groupe algébrique qui est un monoïde.

Pour un monoïde algébrique $S$, on peut donc définir son groupe des inversibles:
\begin{center}
$G(S):=\{x\in S\vert ~\exists !~ y\in S, xy=yx=1\}$.
\end{center}
Si $G(S)$ est connexe réductif, on dit que $S$ est un monoïde \textit{réductif}.
Dans la suite, nous ne considérerons que des monoïdes algébriques irrréductibles et réductifs. 
Soient $S$, $S'$ deux semi-groupes. On appelle $\phi:S\rightarrow S'$ un morphisme de semi-groupes si c'est un morphisme de schémas et que:
\begin{center}
$\forall~ x,y\in S, \phi(xy)=\phi(x)\phi(y)$.
\end{center}
Si de plus, $S$ et $S'$ sont des monoïdes et $\phi(1_{S})=1_{S'}$, on parle de morphisme de monoïdes.
Soit $G_{+}:= (T\times G)/{Z_{G}}$ où le centre $Z_{G}$ de $G$ est plongé par:
\begin{center}
$\lambda.(t,g)=(\lambda t,\lambda^{-1}g)$.
\end{center}
Ce groupe admet un tore maximal $T_{+}=(T\times T)/{Z_{G}}$ dont le groupe des caractères $X^{*}(T_{+})$ s'identifie à un sous-réseau de  $X^{*}(T)\times X^{*}(T)$. On note $Z_{+}$ le centre de $G_{+}$, il s'identifie au tore $T$.
Nous avons un schéma fourre-tout: 
\begin{center}
$H_{G}:=\prod\limits_{1\leq i\leq r}\mathbb{A}^{1}_{\alpha_{i}}\times\prod\limits_{1\leq i\leq r}\End(V_{\omega_{i}})$
\end{center}
(resp $H_{G}^{0}:=\prod\limits_{1\leq i\leq r}\mathbb{A}^{1}_{\alpha_{i}}\times\prod\limits_{1\leq i\leq r}(\End(V_{\omega_{i}})-\{0\})$).
On considère alors l'immersion localement fermée:
$$ \begin{array}{ll}
\iota:&G_{+} \rightarrow H_{G} \\
   &(t,g)\rightarrow(\alpha_{i}(t),\omega_{i}(t)\rho_{\omega_{i}}(g))_{1\leq i\leq r}
\end{array}.$$
On définit alors $V_{G}$, la normalisation de l'adhérence de $G_{+}$ dans le schéma fourre-tout $H_{G}$ et $V_{G}^{0}$ l'image réciproque dans $V_{G}$ de l'adhérence de $G_{+}$ dans $H_{G}^{0}$.
On a les propriétés suivantes sur les adhérences:
\begin{prop}
\begin{itemize}
\item Le schéma $V_{G}$ est muni d'une structure de monoïde.  Son groupe des unités s'identifie à $G_{+}$ et c'est un schéma normal affine et intègre. On l'appelle le semi-groupe de Vinberg.
\item L'action de $G_{+}\times G_{+}$ sur $G_{+}$, donnée par la translation à gauche et à droite s'étend à $V_{G}$ et à $V_{G}^{0}$.
\end{itemize}
\end{prop}
\begin{proof}
Prouvons la première assertion. Tout d'abord, le morphisme $\iota$ est un morphisme de monoïdes, donc l'adhérence de $G_{+}$ dans $H_{G}$ a une structure de monoïde.
Maintenant, d'après un lemme de Renner \cite[Lem.1]{Ri}, si $S$ est un monoïde intègre, il existe une unique structure de monoïde sur sa normalisation $S_{norm}$ telle que $S_{norm}\rightarrow S$ soit un morphisme de monoïdes et telle que $G(S_{norm})=G(S)$. En particulier, en l'appliquant à $V_{G}$, on a le résultat souhaité.

Pour la deuxième assertion, à nouveau, par \cite[Thm. 3]{Ri}, nous avons que l'action de $G_{+}\times G_{+}$ s'étend à l'adhérence  de $G_{+}$ dans $H_{G}$. Puis, d'après \cite[Cor. 2]{Ri}, nous avons une action naturelle de $G_{+}\times G_{+}$ sur $V_{G}$ qui étend l'action par translation à gauche et à droite, compatible au morphisme de normalisation; en particulier, on obtient le résultat analogue pour $V_{G}^{0}$.
\end{proof}
On a la proposition suivante due à Vinberg en caractéristique zéro \cite[Th. 8]{Vi}, que nous étendons en caractéristique $p>0$.
\begin{prop}\label{concini}
Le schéma $V_{G}^{0}$ est lisse.  Il admet un quotient projectif lisse par le centre $Z_{+}$ de $G_{+}$, appelé la compactification de de Concini-Procesi (cf. \cite{dC-P}).
\end{prop}
$\rmq$ Nous reportons la preuve à la section \ref{dcp}.
\medskip
On définit  le schéma $V_{T}$ l'adhérence de $T_{+}$ dans $V_{G}$.
Comme on dispose également de l'adhérence $V_{T}^{\flat}$ de $T_{+}$ dans $H_{G}$, nous cherchons à les comparer.
Si on note $V_{G}^{\flat}$, l'adhérence de $G_{+}$ dans $H_{G}$ et $\zeta:V_{G}\rightarrow V_{G}^{\flat}$ la flèche de normalisation, la flèche $\zeta$ induit un morphisme:
\begin{center}
$\zeta_{T}:V_{T}\rightarrow V_{T}^{\flat}$.
\end{center}

D'après Drinfeld \cite{Dr}, nous avons la proposition suivante:
\begin{prop}\label{drin}
Supposons $\car(k)>3$, alors la flèche $\zeta_{T}$ est un isomorphisme.
\end{prop}
\begin{proof}
Tout d'abord, il résulte de \cite[Cor. 6.2.14]{BK}  que $V_{T}$ est normal.
Le semi-groupe $S$ des poids qui définit $V_{T}^{\flat}$ est engendré par les vecteurs $(\alpha_{i},0)$ et les $(\omega_{i},\lambda)$ pour $\lambda$ un poids de $V_{\omega_{i}}$. Comme $\car(k)>3$, d'après Premet \cite{Pr} et le théorème de Curtis-Steinberg \cite[Th. 3.3]{Bo}, le  semi-groupe $S$ est le même qu'en caractéristique nulle et en caractéristique nulle, $V_{T}^{\flat}$ est normal associé au cône:
\begin{center}
$C^{*}=\{(\lambda,\mu)\in X^{*}(T)\times X^{*}(T)\vert~\lambda\leq\mu\}$ 
\end{center}
En particulier, les schémas  $V_{T}$ et $V_{T}^{\flat}$ s'identifient naturellement via $\zeta_{T}$.
\end{proof}
De plus, si $\overline{Z}_{+}$ est l'adhérence de $Z_{+}$, le centre de $G_{+}$, dans $V_{G}$  en regardant l'algèbre $k[V_{G}]^{G\times G}$, elle s'identifie à $k[\overline{Z}_{+}/Z_{G}]$ et donne lieu à un morphisme de monoïdes, dit d'abélianisation (\cite[Th. 3]{Vi}  et \cite[Th. 20]{Ri}):
\begin{center}
$\alpha:V_{G}\rightarrow A_{G}:=\overline{Z}_{+}/Z_{G}:=\mathbb{A}^{r}$. 
\end{center}
Nous allons maintenant construire un autre morphisme, issu de l'action par conjugaison.
\subsection{Le morphisme de Steinberg étendu $\chi_{+}$}
On rappelle le théorème de Chevalley-Steinberg (\cite[VI.3.1-Ex1]{Bki} et \cite[Th. 6.1]{S} ):
\begin{thm}
On considère l'action adjointe de $G$ sur lui-même. On a l'isomorphisme suivant, qui s'obtient par restriction des fonctions:
\begin{center}
$\phi:k[G]^{G}\stackrel{\cong}{\rightarrow} k[T]^{W}$.
\end{center}
On obtient un morphisme $\chi:G\rightarrow T/W$, dit de Steinberg.
Si de plus, $G$ est simplement connexe, $T/W$ a une structure d'espace affine, donnée par les caractères $\chi_{\omega_{i}}$ pour $1\leq i\leq r$.
\end{thm}
On dispose alors d'un morphisme:
\begin{center}
$\chi_{+}: G_{+}\rightarrow T_{+}/W=\mathbb{G}_{m}^{r}\times \mathbb{A}^{r}$,
\end{center}
provenant du théorème de Chevalley où $W$ agit trivialement sur le tore central.
Le morphisme $\chi_{+}$ est donné par: 
\begin{center}
$g_{+}=(t,g)\rightarrow (\alpha_{1}(t),..,\alpha_{r}(t),\chi_{1}(tg),..,\chi_{r}(tg))$
\end{center}
où  $\chi_{i}:=\Tr(\rho_{(\omega_{i},\omega_{i})})$.
Montrons que cette flèche se prolonge au semi-groupe de Vinberg.
\begin{prop}\label{chevalley}
L'application de restriction $\phi:k[V_{G}]^{G}\rightarrow k[V_{T}]^{W}$ est un isomorphisme de $k$-algèbres. De plus, $V_{T}/W$ est un espace affine de dimension $2r$, dont les coordonnées sont données par les $(\alpha_{i},0)$ et les $\chi_{i}=\Tr(\rho_{(\omega_{i}, \omega_{i})})$.
\end{prop}
\begin{proof}
L'injectivité résulte du fait que si $\phi(f)=\phi(g)$, alors $f=g$ sur l'ensemble $G_{+}^{ss}$ des éléments semi-simples qui est dense dans $G_{+}$ et donc dans 
$V_{G}$ et $f=g$ car $V_{G}$ affine.

Il reste à montrer que $\phi$ est surjective, en particulier, $k[V_{T}]^{W}$ est engendré par les caractères des représentations. Dans la proposition \ref{drin}, nous avons vu que le monoïde $C^{*}\cap X^{*}(T_{+})^{+}$ est engendré par les $(\alpha_{i}, 0)$ et $(\omega_{i},\omega_{i})$. Comme $W$ agit trivialement sur le premier facteur, on obtient que $k[V_{T}]^{W}$ est engendré par les éléments de la forme $(\gamma, 0)$ et $(\la, \Sym \la)$, pour $\gamma$ dans le module engendré par les racines simples et $\lambda\leq\omega_{i}$. Ici, nous avons posé $\Sym\la:=\sum\limits_{w\in W}w\la$.

Maintenant, les éléments $(\omega_{i},\chi_{\omega_{i}})$ sont égaux à $(\omega_{i}, \Sym(\omega_{i}))$ plus une somme d'éléments $(\omega_{i}, \Sym \mu$) pour $\mu<\omega_{i}$.
Or, comme $\mu<\omega_{i}$, la différence est une somme de racines simples et $(\omega_{i},\Sym \mu)=(\omega_{i}-\mu,0)+(\mu, \Sym \mu)$. On peut donc remplacer les $(\lambda, \Sym\lambda)$ par les $(\omega_{i},\chi_{\omega_{i}})$.
Enfin, d'après Steinberg \cite[Thm. 6.1]{S}, ces éléments sont algébriquement indépendants et forment une base.
\end{proof}
On obtient  un nouveau morphisme que l'on note $\chi_{+}$, étendant le précédent:
\begin{center}
$\chi_{+}:V_{G}\rightarrow V_{T}/W$.
\end{center}
Si $G$ est simplement connexe, Steinberg  construit dans \cite[§7]{S}  un morphisme:
\begin{center}
 $\epsilon:T/W\rightarrow G$ 
\end{center}
qui est une section au morphisme $\chi$.
Pour un $r$-uplet $(a_{1},.., a_{r})\in T/W:=\mathbb{A}^{r}$, elle se définit de la manière suivante:
\begin{center}
$\epsilon(a_{1},.., a_{r}):=\prod\limits_{i=1}^{r}x_{\alpha_{i}}(a_{i})n_{i}$,
\end{center}
où les $x_{\alpha_{i}}(a_{i})$ sont des éléments du groupe radiciel $U_{\alpha_{i}}$ et les $n_{i}$ sont des éléments du normalisateur $N_{G}(T)$ représentant les réflexions simples $s_{\alpha_{i}}$ de $W$.\\
Ainsi, $\epsilon(a)\in\prod\limits_{i=1}^{r}U_{i}n_{i}$ et en utilisant les relations de commutation on a que :
\begin{center}
$\prod\limits_{i=1}^{r}U_{i}n_{i}=U_{w}w$
\end{center}
où $w=s_{1}s_{2}...s_{r}$ et $U_{w}=U\cap wU^{-}w^{-1}$.
Steinberg établit que cette section est à valeurs dans l'ouvert $G^{reg}:=\{g\in G\vert~ \dim I_{g}=r\}$, où $I_{g}$ désigne le centralisateur de $g$. Il prouve également grâce à son critère différentiel, les propriétés suivantes \cite[Th. 8.1]{S} :
\begin{thm}
\begin{enumerate}
\item Le morphisme $\chi^{reg}:G^{reg}\rightarrow T/W$ est lisse.
\item Tout élément de $G^{reg}$ est conjugué à un élément de $\epsilon(T/W)$.
\end{enumerate}
\end{thm}

En fait, pour chaque élément de Coxeter $w'\in W$, Steinberg définit une autre section $\epsilon_{+}^{w'}$ qui est conjuguée à la précédente \cite[Lem. 7.5]{S} . 
En revanche, on verra que sur le semi-groupe de Vinberg, elles ne le sont plus.
Nous donnons donc une définition, puisque par la suite, ces sections vont apparaître de manière cruciale.
Pour $w'\in W$ un élément de Coxeter, on écrit une décomposition réduite de $w'$: 
\begin{center}
$w'=s_{\alpha_{i_{1}}}...s_{\alpha_{i_{r}}}$.
\end{center}
Et la section $\epsilon_{+}^{w'}$ se définit de la même manière 
\begin{center}
$\prod\limits_{i=1}^{r}U_{i_{k}}n_{i_{k}}=U_{w'}w'$.
\end{center}
Les résultats établis par la suite pour l'élément de Coxeter particulier $w=s_{1}...s_{r}$, admettront la même démonstration pour un autre élément de Coxeter $w'\in W$.
L'objectif est maintenant d'étendre cette section au semi-groupe de Vinberg ainsi que les résultats précédemment établis par Steinberg.
\subsection{Prolongement de la section de Steinberg}
En s'inspirant de la formule de Steinberg, on considère donc:
\begin{center}
$(b,a):=(b_{1},.., b_{r}, a_{1},.., a_{r}) \in \mathbb{G}_{m}^{r}\times \mathbb{A}^{r}$.
\end{center}
Soit la flèche $\phi:T\rightarrow T_{+}$, $t\rightarrow (t,t^{-1})$. L'image $T_{\Delta}$, le tore antidiagonal est isomorphe à $T^{ad}$, et on a un isomorphisme canonique donné par:
\begin{center}
$\alpha_{\bullet}:T_{\Delta}\rightarrow\mathbb{G}_{m}^{r}$.
\end{center}
Soit alors $\psi$ l'isomorphisme inverse, nous optons pour la notation indiciaire. Ainsi pour $b\in \mathbb{G}_{m}^{r}$, on a: 
\begin{center}
$\forall~ 1\leq i\leq r, \alpha_{i}(\psi_{b})=b_{i}$ et $\psi_{b}\in T_{\Delta}$.
\end{center}
On définit $\epsilon_{+}:\mathbb{G}_{m}^{r}\times \mathbb{A}^{r}\rightarrow G_{+}$ par: 
\begin{center}
$\epsilon_{+}(b,a)=\epsilon(a)\psi_{b}$
\end{center}
où $\epsilon$ est la section de Steinberg pour le groupe $G$ définie par $\epsilon(a):=\prod\limits_{i=1}^{r}x_{i}(a_{i})n_{i}$.
Commençons par rappeler la manière dont agissent les différents éléments considérés. Soit $V_{\omega}$ une représentation irréductible de plus haut poids $\omega$. Elle admet une décomposition en espace de poids:
\begin{center}
$V_{\omega}=\bigoplus\limits_{\nu\leq\omega}V(\nu)$.
\end{center}
On a d'après  \cite[Lem. 7.15]{S} :
\begin{lem}\label{rep}
\begin{enumerate}
\item Si $U_{\alpha}=\{x_{\alpha}(c)\}$ désigne le groupe radiciel associé à $\alpha$, alors l'action de $U_{\alpha}$ sur un vecteur de poids $\mu$, $v_{\mu}$ est de la forme:
\begin{center}
$x_{\alpha}(c).v_{\mu}=\sum\limits_{k\geq 0}c^{k}v_{k}$ 
\end{center}
où $v_{k}\in V(\mu+k\alpha)$, est indépendant de $c$.
\item Pour tout $i,j\in \{1,\dots,r\}$, nous avons $s_{j}\omega_{i}=\omega_{i}-\delta_{ij}\alpha_{j}$ et pour $\lambda=\sum m_{i}\omega_{i}$, $s_{j}\lambda=\lambda-m_{j}
\alpha_{j}$.
Enfin, nous introduisons un ordre partiel strict sur l'ensemble $\{1,\dots,r\}$; si $i\neq j$, on écrit $j\prec i$ s'il existe un poids dominant $\lambda=\sum m_{i}\omega_{i}$ avec $\lambda<\omega_{i}$ et $m_{j}>0$.
\end{enumerate}
\end{lem}

\begin{prop}\label{semisimple}
Le morphisme $\epsilon_{+}:T_{+}/W:=\mathbb{G}_{m}^{r}\times\mathbb{A}^{r}\rightarrow G_{+}$ est une section de $\chi_{+}$ et se prolonge  en un morphisme, noté de la même manière, $\epsilon_{+}:V_{T}/W\rightarrow V_{G}$, à valeurs dans $V_{G}^{0}$. 
\end{prop}
\begin{proof}
Fixons $i$ et considérons la représentation irréductible de plus haut poids $\omega_{i}$ et $\lambda\leq\omega_{i}$ un poids de la représentation. Un élément $g_{+}=(t,g)$ de $G_{+}$ agit par $\omega_{i}(t)\rho_{\omega_{i}}(g)$. Soit $v_{\la}$ un vecteur de poids $\lambda$ et regardons comment agit l'élément $\epsilon_{+}(b,a)$.  Nous avons la formule suivante:
\begin{equation}
\epsilon_{+}(b,a)v_{\lambda}=\epsilon(a)\psi_{b}v_{\lambda}=\omega_{i}(\psi_{b})\lambda(\psi_{b})^{-1}\epsilon(a)v_{\lambda}=(\omega_{i}-\lambda)(\psi_{b})\epsilon(a)v_{\lambda}.
\label{etoile}
\end{equation}
avec $\psi_{b}\in T_{\Delta}$.
Comme $\omega_{i}-\lambda$ est une somme de racines simples, on en déduit que $(\omega_{i}-\lambda)(\psi_{b})$ est polynômial en les $b_{i}$, donc $\epsilon_{+}$ se prolonge.

Montrons que le prolongement est non nul. On considère le vecteur de plus haut poids $v_{\omega_{i}}$. En ce cas, $\psi_{b}$ agit trivialement, on a alors la formule suivante:
\begin{center}
$\epsilon_{+}(b,a)v_{\omega_{i}}=\epsilon(a)v_{\omega_{i}}$.
\end{center}
Comme $\eps(a)\in G$, on a bien $\eps(a)v_{\omega_{i}}\neq0$.
Montrons que c'est une section. Du lemme $\ref{rep}$ et de \eqref{etoile}, on déduit la formule suivante:
\begin{center}
$\epsilon_{+}(b,a)v_{\lambda}=(\omega_{i}-\lambda)(\psi_{b})\sum\limits_{k_{j\geq 0}}a_{1}^{k_{1}}...a_{r}^{k_{r}}v_{(k_{1},...,k_{r})}$
\end{center}
où $v_{(k_{1},...,k_{r})}$ est un vecteur de poids $\lambda +\sum(k_{j}-m_{j})\alpha_{j}$ et est indépendant des $a_{j}$. Examinons la contribution à la trace, qui est non nulle que si $m_{j}=k_{j}\geq 0$. On distingue  deux cas:
\begin{itemize}
\item
$\lambda=\omega_{i}$ et la contribution est $a_{i}$.
\item
$\lambda\neq\omega_{i}$, et on a $\lambda<\omega_{i}$.
\end{itemize}
La formule montre que $a_{j}$ n'apparaît pas si $m_{j}=k_{j}=0$, donc seulement les $c_{j}$ avec $\omega_{j}$ dans le support de $\lambda$ (et donc $j\neq i$), ont une contribution non nulle.

Comme $\lambda$ est dominant et $\lambda<\omega_{i}$, nous avons $j\prec i$. Et donc le vecteur $v_{\lambda}$ contribue à la trace  par un polynôme en les $a_{j}$ avec $j\prec i$.
En conclusion, $\chi_{i}$ est un polynôme en les $a_{i}$, de la forme $a_{i}+$ termes en les $a_{j}$, pour $j\prec i$. Donc, à partir des $\chi_{i}$, nous récupérons bien les paramètres $a_{i}$, comme souhaité.
\end{proof}

Nous avons une action \textit{canonique} de $Z_{+}$ sur $V_{T}/W=\mathbb{A}^{r}\times\mathbb{A}^{r}$ donnée par:
\begin{center}
$z.(b_{1},..., b_{r}, a_{1},..., a_{r})=(\alpha_{1}(z)b_{1},...,\alpha_{r}(z)b_{r},\omega_{1}(z)a_{1},...,\omega_{r}(z)a_{r})$.
\end{center}
En général la section $\eps_{+}$ n'est jamais $Z_{+}$-équivariante.
Il nous faut donc tordre judicieusement l'action de $Z_{+}$ pour avoir une section équivariante. Une légère complication que l'on retrouve également dans le cas de la section de Kostant pour les algèbres de Lie, est qu'il faut extraire des racines.
Dans la suite, on pose $c:=\left|Z_{G}\right|$.

Pour $z\in Z_{+}$, on considère $\omega_{\bullet}(z):=(\omega_{1}(z),..,\omega_{r}(z))\in \mathbb{G}_{m}^{r}$. On a alors l'élément:
\begin{equation}
\psi_{\omega_{\bullet}(z)}:=(\lambda_{z},\lambda_{z}^{-1})\in T_{\Delta} ~~\text{tel que}~~  \alpha_{i}(\psi_{\omega_{\bullet}(z)})=\omega_{i}(z)^{c}, \text{pour tout}~ i=1,..,r.
\label{psibul}
\end{equation}
$\rmq$ Si on n'élève pas à la puissance $c$, l'élément $\psi_{\omega_{\bullet}(z)}$ n'est pas défini de manière univoque.
 
On construit donc par ce procédé un morphisme de groupes $\tau=\psi\circ\omega:Z_{+}\rightarrow T_{\Delta}$.
On considère maintenant la nouvelle action $*$ de $Z_{+}$ sur $G_{+}$ donnée par:
\begin{center}
$z*g=z^{c}\psi_{\omega_{\bullet}(z)}^{-1}g\psi_{\omega_{\bullet}(z)}$.
\end{center}
et on définit également sur $\kc$ une nouvelle action $*$ de $Z_{+}$ définie par :
\begin{center}
$z*(b, a)=z^{c}.(b, a)$.
\end{center}
Le morphisme de Steinberg $\chi_{+}$ reste bien $Z_{+}$-équivariant pour cette action.
On considère l'élément:
\begin{equation}
\psi_{\omega_{\bullet}(z)}:=(\lambda_{z},\lambda_{z}^{-1})\in T_{\Delta}~~ \text{tel que}~~ \alpha_{i}(\psi_{\omega_{\bullet}(z)})=\omega_{i}(z)^{c},  \text{pour tout}~ i=1,..,r.
\label{albul}
\end{equation}
Nous avons alors la proposition suivante concernant la section $\epsilon_{+}$.
\begin{prop}\label{action}
La section de Steinberg $\epsilon_{+}:V_{T}/W\rightarrow V_{G}^{0}$ est $Z_{+}$-équivariante pour l'action $*$ sur $G_{+}$ et sur $V_{T}/W$.
\end{prop}
\begin{proof}
Il nous faut prouver l'identité $\epsilon_{+}(z^{c}.(b,a))=z*\epsilon_{+}(b,a)$. Pour simplifier, nous supposerons $c=1$, la preuve étant exactement la même.
Il nous suffit de démontrer cette identité au-dessus de $G_{+}$. En rappelant que $\epsilon_{+}(b,a)=\epsilon(a)\psi_{b}$ où $\epsilon$ est la section de Steinberg pour $G$ et $\psi_{b}\in T_{\Delta}$, l'identité se réduit à:
\begin{center}
$\epsilon(z.a)\psi_{\alpha_{\bullet}(z)b}=z\psi_{\omega_{\bullet}(z)}^{-1}\epsilon(a)\psi_{b}\psi_{\omega_{\bullet}(z)}$.
\end{center}
qui se simplifie en:
\begin{center}
$\epsilon(z.a)\psi_{\alpha_{\bullet}(z)}=z\psi_{\omega_{\bullet}(z)}^{-1}\epsilon(a)\psi_{\omega_{\bullet}(z)}$.
\end{center}
Considérons la représentation fondamentale de poids $\omega_{i}$.
Soit $\mu\leq\omega_{i}$ un poids de la représentation et un vecteur de poids $\mu$, $v_{\mu}$.
On regarde l'action de $\psi_{\alpha_{\bullet}(z)}$ sur ce vecteur.
Si on écrit $\psi_{\alpha_{\bullet}(z)}:=(\gamma_{z},\gamma_{z}^{-1})$, on obtient:
\begin{equation}
\psi_{\alpha_{\bullet}(z)}v_{\mu}=(\omega_{i}-\mu)(\gamma_{z})v_{\mu}=(\omega_{i}-\mu)(z)v_{\mu} 
\label{precp}
\end{equation}
car $\omega_{i}-\mu$ est combinaison linéaire de racines simples et par \eqref{albul}.
En utilisant le lemme $\ref{rep}$ et d'après \eqref{precp}, on déduit que:
\begin{center}
$(\epsilon(z.a)\psi_{\alpha_{\bullet}(z)}).v_{\mu}=(\omega_{i}-\mu)(z)\sum\limits_{k_{j\geq 0}}(\omega_{1}(z)a_{1})^{k_{1}}...(\omega_{r}(z)a_{r}^{k_{r}})v_{(k_{1},...,k_{r})}$,
\end{center}
où $v_{(k_{1},...,k_{r})}$ est un vecteur de poids $\mu +\sum\limits_{j=1}^{r}(k_{j}-m_{j})\alpha_{j}$ et $\mu=\sum\limits_{j=1}^{r} m_{j}\omega_{j}$.
D'autre part, nous avons également:
\begin{equation}
z\psi_{\omega_{\bullet}(z)}^{-1}\epsilon(a)\psi_{\omega_{\bullet}(z)}v_{\mu}=\omega_{i}(z)\mu(\la_{z}^{-1})\sum\limits_{k_{j\geq 0}}[\mu +\sum\limits_{j=1}^{r}(k_{j}-m_{j})\alpha_{j}](\lambda_{z})a_{1}^{k_{1}}...a_{r}^{k_{r}}v_{(k_{1},...,k_{r})}.
\label{membg}
\end{equation}
où l'on rappelle que $\psi_{\omega_{\bullet}(z)}=(\lambda_{z},\lambda_{z}^{-1})$.
Maintenant, comme $\mu=\sum\limits_{j=1}^{r}m_{j}\omega_{j}$, en utilisant \eqref{psibul}, nous avons l'égalité suivante: 
\begin{center}
$[\mu+\sum(k_{j}-m_{j})\alpha_{j}](\la_{z})=\mu(\la_{z})\mu(z^{-1})\prod\limits_{j=1}^{r}k_{j}\alpha_{j}(\la_{z})=\mu(\la_{z})\mu(z^{-1})\prod\limits_{j=1}^{r}\omega_{j}(z)^{k_{j}}$.
\end{center}
En injectant cette identité dans \eqref{membg}, nous avons:
\begin{center}
$z\psi_{\omega_{\bullet}(z)}^{-1}\epsilon(a)\psi_{\omega_{\bullet}(z)}v_{\mu}=\omega_{i}(z)\mu(\la_{z}^{-1})\sum\limits_{k_{j\geq 0}}[\mu +\sum\limits_{j=1}^{r}(k_{j}-m_{j})\alpha_{j}](\lambda_{z})a_{1}^{k_{1}}...a_{r}^{k_{r}}v_{(k_{1},...,k_{r})}$.
\end{center}
Et donc, les deux membres sont égaux, ce qu'on voulait.
\end{proof}
On a une description plus agréable du prolongement. Soit le diagramme commutatif suivant:
$$\xymatrix{T_{\Delta}\ar[d]^{(\alpha,0)}\ar[rrr]^-{((\alpha,0),\omega\rho_{\omega})}&&&\prod\limits_{1\leq i\leq r}\mathbb{A}^{1}_{\alpha_{i}}\times\prod\limits_{1\leq i\leq r}\End(V_{\omega_{i}})\ar[d]^{p_{1}}\\\mathbb{G}_{m}^{r}\ar[rrr]&&&\mathbb{A}^{r}}$$
où la flèche verticale de gauche est un isomorphisme et la flèche horizontale du bas, l'injection canonique.

Pour chaque $i$, la matrice $\omega_{i}(t)\rho_{\omega_{i}}(t^{-1})$ est un polynôme en les $\alpha_{i}(t)$, et ces polynômes sont indépendants de $t$, on obtient alors que l'isomorphisme entre $T_{\Delta}$ et $\mathbb{G}_{m}^{r}$ se prolonge en un morphisme:
\begin{center}
$p_{1}:\overline{T}_{\Delta}\rightarrow\mathbb{A}^{r}$, où $\overline{T}_{\Delta}$ est l'adhérence de $T_{\Delta}$ dans $V_{G}$.
\end{center}
Les éléments de $T_{\Delta}$ sont entièrement déterminés par la composante suivant $\prod\limits_{1\leq i\leq r}\mathbb{A}^{1}_{\alpha_{i}}$, puisque leur seconde composante est une matrice dont les coefficients sont des polynômes en les $\alpha_{i}$. Cette flèche est donc un isomorphisme.
Enfin, on remarque que $\overline{T}_{\Delta}$ est en fait contenu dans $V_{G}^{0}$ puisque le coefficient de la matrice $\omega_{i}(t)\rho_{\omega_{i}}(t^{-1})$ en le vecteur de plus haut poids est 1.

On continue de noter $\psi$ l'isomorphisme inverse.
On commence par rappeler la proposition suivante, due à Vinberg \cite[Th. 6-7]{Vi}  et Rittatore \cite[Th. 21]{Ri}:
\begin{prop}\label{strates}
\begin{enumerate}
\item On a une décomposition de $V_{G}^{0}$ en $G_{+}\times G_{+}$-orbites indexées par les sous-ensembles $J$ de $\Delta$.
\item Dans chaque orbite $\mathcal{O}_{J}$, on a un idempotent distingué $e_{J}\in \overline{T}_{\Delta}$.
\item Pour $J\subset \Delta$ et $V_{\lambda}$ la représentation irréductible de poids $\lambda$ de $G$, on considère $V_{J,\lambda}$ le sous-espace engendré par les sous-espaces de poids qui sont dans $\lambda+D_{J}$, $D_{J}$ le module engendré par les $\{\alpha_{j}, j\in J\}$.

Alors, $\rho_{\lambda}(e_{J})=p_{J}^{\lambda}$ où $p_{J}^{\lambda}$ désigne la projection sur le sous-espace vectoriel $V_{J,\lambda}$.\\
En particulier, l'orbite $\mathcal{O}_{J}$ s'envoie sur la strate $\mathbb{G}_{m}^{J}\times\mathbb{A}^{r}$ par $\chi_{+}$.
\item On a également une décomposition de $V_{G}$ de $G_{+}\times G_{+}$-orbites indexées par certaines paires $(J,K)$ de $\Delta\times \Delta$. A nouveau dans chaque orbite $\mathcal{O}_{J,K}$, on a un idempotent distingué $e_{J,K}$ de $V_{T}$ et dans le cas où $J$ est vide, $\rho_{\lambda}(e_{\emptyset,K})=p_{\emptyset}^{\lambda}$ si et seulement si $\lambda\in C_{K}$ et zéro sinon. Ici, $C_{K}$ désigne le cône engendré par les $(\omega_{k},\omega_{k}), k\in K$.
\end{enumerate}

\end{prop}
De cette proposition, on déduit une stratification de $\overline{T}_{\Delta}=\coprod\limits_{J}T_{\Delta}e_{J}$ et chaque strate $T_{\Delta}e_{J}$ s'envoie sur $\mathbb{G}_{m}^{J}\times\mathbb{A}^{r}$.
\begin{lem}
Soit $(b,a)\in \mathbb{A}^{r}\times\mathbb{A}^{r}$, on considère l'élément $\psi_{b}$ de $\overline{T}_{\Delta}$, donné par l'isomorphisme ci-dessus, alors nous avons:
\begin{center}
$\epsilon_{+}(b,a)=\epsilon(a)\psi_{b}$,
\end{center}
où l'on fait la multiplication dans le semi-groupe de Vinberg.

En particulier si on note $C:=\epsilon_{+}(\{1\}\times\mathbb{A}^{r})$, alors $\epsilon_{+}(\mathbb{A}^{r}\times\mathbb{A}^{r})=C\overline{T}_{\Delta}$ et\\ $\epsilon_{+}(\mathbb{G}_{m}^{J}\times\mathbb{A}^{r})=CT_{\Delta}e_{J}$.
\end{lem}
\begin{proof}
Au-dessus de $\mathbb{G}_{m}^{r}\times\mathbb{A}^{r}$, les deux sections sont les mêmes par définition, donc par unicité du prolongement, l'égalité reste vraie sur $\mathbb{A}^{r}\times\mathbb{A}^{r}$. La suite du lemme vient de la description en strates rappelée ci-dessus.
\end{proof}

$\rmq$ Dans le cas $A_{r}$, le calcul de la section de Steinberg nous donne :
  \begin{center}
$\epsilon_{+}(\alpha_{\bullet},a_{\bullet})= \begin{pmatrix}
  a_{1}& -\alpha_{1}a_{2}& \alpha_{1}\alpha_{2}a_{3}&\cdots&(-1)^{r-1}\prod\limits_{i=1}^{r-1}\alpha_{i}a_{r}&(-1)^{r}\prod\limits_{i=1}^{r}\alpha_{i}\\ 1&0&0&\cdots&0&0\\0&\alpha_{1}&0&\dots&0&0\\\vdots&&\ddots&\\\vdots&&&\ddots\\0&0&\dots&\dots
&\prod\limits_{i=1}^{r-1}\alpha_{i}&0\\ 
  \end{pmatrix}.$
  \end{center}
\subsection{Construction du centralisateur régulier}
Dans la suite, on note $\mathfrak{C}_{+}:=V_{T}/W$. Soit le schéma en groupes $I$ sur $V_{G}$ des paires
\begin{center}
$I:=\{(g,\gamma)\in G\times V_{G}\vert~ g\gamma g^{-1}=\gamma\}$.
\end{center}
On définit le centralisateur régulier par $J=\epsilon_{+}^{*}I$. 
Nous voulons montrer que ce schéma en groupes est lisse commutatif de dimension $r$. On commence par montrer que la section tombe dans l'ouvert régulier. Le procédé utilisé sert pour toute la suite de l'article; à savoir, démontrer une propriété pour la fibre la plus singulière et la propager à toutes les autres. 
On rappelle le résultat suivant qui est un corollaire de \cite[VI B. 4, Prop. 4.1]{SGA3}. 
\begin{prop}\label{semicont}
Soient $a, a' \in\mathfrak{C}_{+}$ tels que $a\in\overline{\{a'\}}$, alors $\dime J_{a}\geq \dime J_{a'}$.
En particulier, en appliquant l'inégalité au point générique, on a:  
\begin{center}
$\forall~ a\in \mathfrak{C}_{+}, \dime J_{a}\geq r$.
\end{center}
Ainsi, l'ensemble $\{a\in\mathfrak{C}_{+}\vert~ \dime J_{a}=r\}$ est ouvert.
\end{prop}
Nous avons besoin d'une description plus précise des strates avant d'aller plus avant, on utilise pour cela la proposition tirée de Vinberg \cite[Th. 7]{Vi}  et étendue par Rittatore \cite[Th. 21]{Ri} :
\begin{prop}\label{stabet}
Pour un sous-ensemble $J\subset\Delta$, soit l'orbite $\mathcal{O}_{J}$, on considère les sous-groupes paraboliques $P_{J}$ et $P_{J}^{-}$.
On a une décomposition de Lévi $P_{J}=L_{J}R_{u}(P_{J})$ et on note $\delta$ (resp $\delta_{-}$) la projection de $P_{J}$ sur $L_{J}$ (resp $P_{J}^{-}$ sur $L_{J}$).
L'orbite $\mathcal{O}_{J}$ a un idempotent distingué $e_{J}$. Son stabilisateur $H_{J}$ s'identifie à :
\begin{center}
$H_{J}=\{(x,y)\in P_{J}\times P_{J}^{-}\vert~ \delta(x)\delta_{-}(y)^{-1}\in T_{J,\Delta}\}$, où $T_{J,\Delta}=\{t\in T_{\Delta}\vert~\alpha_{j}(t)=1, j\in J\}$.
\end{center}
\end{prop}

\begin{prop}\label{centreg}
On définit l'ouvert régulier $V_{G}^{reg}:=\{g\in V_{G}\vert~\dime I_{g}=r\}\subset V_{G}^{0}$.
La section $\epsilon_{+}:\mathfrak{C}_{+}\rightarrow V_{G}^{0}$ est à valeurs dans $V_{G}^{reg}$, i.e. le centralisateur régulier est de dimension $r$.
\end{prop}
On commence d'abord par montrer la proposition pour le point le plus singulier:
\begin{lem}\label{cent0}
L'élément $\epsilon_{+}(0)$ appartient à $V_{G}^{reg}$.
\end{lem}
\begin{proof}
On a  vu que la section tombait dans $V_{G}^{0}$. On a $\epsilon_{+}(0)=we_{\emptyset}$, où $w=s_{1}\dots s_{r}$. Soit alors $g\in G$ un élément du centralisateur, alors
\begin{center}
$gwe_{\emptyset}g^{-1}=we_{\emptyset}$.
\end{center}
De la description du stabilisateur des $e_{I}$ dans la proposition $\ref{strates}$, on en déduit que:
\begin{center}
$g\in B^{-}\cap wBw^{-1}$ et que $\delta_{-}(g)=\delta(w^{-1}gw)\in T_{+}$. 
\end{center}
Regardons la seconde condition.
On pose alors $g_{1}=w^{-1}gw\in B$, on écrit $g_{1}=tu$ avec $t\in T$, $u\in U$. Etudions alors ce que vaut $\delta(wg_{1}w^{-1})$:
\begin{center}
$wg_{1}w^{-1}=wtw^{-1}ww^{-1}=[\Ad(w)t][\Ad(w)u]$.
\end{center}
Comme l'élément $w$ agit sur les groupes radiciels $U_{\alpha}$ par $wU_{\alpha}w^{-1}=U_{w\alpha}$, on obtient que $wuw^{-1}$ ne contribue pas dans la projection sur le tore.

Regardons donc l'autre partie. L'égalité $\delta_{-}(g_{1})=\delta(wg_{1}w^{-1})$ impose que $t=t_{w}$  ce qui donne que pour tout $i$, $\alpha_{i}(t)=1$ et $t\in Z_{G}$, qui est fini. Le centralisateur s'identifie alors à:
\begin{center}
$J_{0}=Z_{G}(U^{-}\cap wUw^{-1})$
\end{center}
Maintenant, on sait que $U^{-}\cap wUw^{-1}$  s'identifie au produit des groupes radiciels $U_{\alpha}$, avec $\alpha>0$ et $w\alpha<0$ et est de dimension $l(w)=r$.
\end{proof}
Passons à la preuve de la proposition \ref{centreg}, la méthode de démonstration sera ensuite réutilisée systématiquement.

\begin{proof}
On regarde l'ensemble $U:=\{a\in\mathfrak{C}_{+}\vert~ \dime J_{a}=r\}$. C'est un ouvert d'après la proposition $\ref{semicont}$ et d'après ci-dessus, il contient $0$. Il est de plus $Z_{+}$-équivariant. Montrons que $U=\mathfrak{C}_{+}$. Supposons par l'absurde que le complémentaire de $U$ soit non vide.
Soit $a$ un point du complémentaire. On considère l'adhérence de sa $Z_{+}$-orbite que l'on note $F$. Celle-ci contient $0$ qui est dans $U$, donc on en déduit
que c'est aussi le cas du point générique de $F$ et comme $U$ est $Z_{+}$-équivariant, on obtient $a\in U$, une contradiction.
\end{proof}
$\rmq\label{principe}$ L'argument ci-dessus montre que toute propriété (P) sur les fibres, vérifiée par $\chi_{+}^{-1}(0)$, qui est ouverte et qui est $Z_{+}$-équivariante, se propage à toutes les fibres $\chi_{+}^{-1}(a)$.
Cette remarque sera utilisée de manière systématique par la suite.

Il nous faut encore obtenir que la flèche $\chi_{+}^{reg}$ est lisse. Nous allons étudier plus en détail la fibre au-dessus de 0 et ensuite utiliser la remarque ci-dessus.

\section{Propriétés du morphisme de Steinberg étendu $\chi_{+}$}
\subsection{Etude du cône nilpotent}
\begin{prop}
Soit la flèche $\chi_{+}:V_{G}\rightarrow\mathfrak{C}_{+}$.
Alors, $\mathcal{N}:=\chi_{+}^{-1}(0)$ est inclus dans $\{g\in V_{G} \vert~\forall~ \lambda, \rho_{\lambda}(g) ~\text{est nilpotente}\}$, en particulier dans les représentations de dimension un, $\rho_{\lambda}(g)$ est nulle.
\end{prop}
\begin{proof}
Soit $\gamma\in V_{G}$ tel que $\chi_{+}(\gamma)=0$ alors $\gamma$ est dans l'union des strates $\coprod\limits_{J\subset\Delta}\mathcal{O}_{\emptyset,J}$. Il s'écrit donc $\gamma=g_{1}e_{\emptyset,J}g_{2}$. Quitte à conjuguer, on l'écrit sous la forme $ge_{\emptyset,J}$. De plus, on peut se ramener à $J=\Delta$ puisqu'il résulte de la description de la proposition $\ref{strates}$ (iii), que pour $J\neq\Delta$, $e_{\emptyset,J}$ est l'endomorphisme nul si on regarde à travers une représentation $\rho_{\lambda}$ avec $\lambda\notin C_{J}$.

Regardons comment cet élément agit  sur la représentation irréductible de plus haut poids $\omega_{i}$. Soit $v_{\lambda}$ un vecteur de poids $\lambda$. Alors, toujours par la proposition $\ref{strates}$ (iii), on déduit que:
\begin{center}
$ge_{\emptyset}v_{\lambda}=0$ si $\lambda\neq\omega_{i}$
\end{center}
et $ge_{\emptyset}v_{\omega_{i}}=a_{i}v_{\omega_{i}}+$ (termes liés à d'autres vecteurs).
Or, comme $\chi_{+}(\gamma)=0$, on en déduit que pour tout $i$, $a_{i}=0$. Ainsi, on obtient que $\rho_{\omega_{i}}(\gamma)$ est une matrice nilpotente pour tout $i$. De plus, pour les représentations de dimension un, le fait que $\gamma$ est dans la strate indexée par le vide, nous assure que  c'est l'endomorphisme nul pour de telles représentations.
Comme $V_{G}$ est fermé dans un certain $\End(V)$, $\mathcal{N}=\Nil(V)\cap V_{G}$ où $\Nil(V)$ est le cône nilpotent. Comme il est fermé, on a le résultat analogue pour $\mathcal{N}$.
\end{proof}
On s'intéresse à l'ouvert du cône nilpotent :
\begin{center}
$\mathcal{N}^{0}=\mathcal{N}\cap V_{G}^{0}$.
\end{center}
En effet, le lien avec la compactification magnifique de de Concini-Procesi (cf. Prop. $\ref{concini}$) nous permet de déterminer sa structure. Nous allons avoir besoin d'introduire un certain nombre d'objets, en particulier la compactification magnifique.

Pour un sous-ensemble $I\subset \Delta$, considérons $W_{I}$ le sous-groupe de $W$ engendré par les réflexions simples $s_{i}$, $i\in I$. Soit $W^{I}$ l'ensemble des représentants de longueur minimale de $W/W_{I}$.

On note $X$ la compactification magnifique de de Concini Procesi. Dans \cite{dC-P}, on établit les propriétés suivantes concernant la compactification:
\begin{enumerate}\label{dcp}
\item
C'est une $G_{ad}\times G_{ad}$ variété projective lisse qui contient $G_{ad}$ comme ouvert dense.
\item
Les $G_{ad}\times G_{ad}$-orbites de $X$ sont indexées par les sous-ensembles $I\subset \Delta$.
\item
On a $X_{\Delta}=G_{ad}$ et $X_{\emptyset}$ est l'orbite fermée.
Les orbites $X_{I}$ admettent une description similaire aux orbites de $V_{G}^{0}$.
\end{enumerate}
Nous sommes en mesure de démontrer le lien entre $V_{G}^{0}$ et la compactification magnifique annoncé dans la proposition \ref{concini}:
\begin{prop}
Le schéma $V_{G}^{0}$ est lisse et le quotient  de $V_{G}^{0}$ par l'action par homothétie du centre $Z_{+}$ s'identifie à la compactification magnifique $X$.
\end{prop}
La preuve suit celle de Vinberg en utilisant la proposition \ref{strates}, étendue par Rittatore en caractéristique $p$:
\begin{proof}
On considère la flèche:
\begin{center}
$j: U^{-}\times Z_{+}\times\overline{T}_{\Delta}\times U\rightarrow V_{G}^{0}$
\end{center}
donnée par $(u_{-},z,t,u)\mapsto u_{-}ztu$.
La flèche $j$ est birationnelle au-dessus de $G_{+}$ entre schémas normaux intègres.
Montrons qu'elle est quasi-finie.
On doit prouver qu'une égalité:
\begin{center}
$u_{-}zt_{1}e_{I}u=t_{2}e_{I}, u_{-}\in U^{-}, z\in Z_{+}, t_{1},t_{2}\in T_{\Delta}$
\end{center}
implique $u_{-}=u=z=1$ et $t_{1}=t_{2}$.
Cela revient alors de la description des stabilisateurs \ref{stabet}, de la même manière que \cite[Prop. 14]{Vi}.
Ainsi, par le Main Theorem de Zariski, on en déduit que $j$ est une immersion ouverte. Comme de plus, son image contient des représentants de toutes les $G_{+}\times G_{+}$-orbites de $V_{G}^{0}$, on obtient la lissité.
L'assertion sur le quotient s'obtient alors de la même manière que \cite[Prop. 15]{Vi} .
\end{proof}

Pour un cocaractère $\lambda\in X_{*}(T)$, comme $X$ est projective, on peut définir $\lambda(0)$ par le critère valuatif.
\begin{lem}
Il existe un unique point base $h_{I}\in X_{I}$ tel que pour tout cocaractère $\lambda$ vérifiant $\lambda(\alpha)=0$ pour $\alpha\in I$ et $\lambda(\alpha)>0$ pour $\alpha\in \Delta-I$, on a $\lambda(0)=h_{I}$. De plus, on a
$X_{I}=(G_{ad}\times G_{ad}).h_{I}$.
\end{lem}
On renvoie à \cite[sect. 3]{DS} et \cite[p. 73]{Sp1}  pour la preuve de ce lemme. Les $h_{I}$ sont les images des idempotents $e_{I}$ de $V_{G}^{0}$ de Vinberg.
Suivant Lusztig \cite[12.3]{L2}, , nous allons décrire une certaine partition de $X$.
Pour $I\subset \Delta$ et $w\in W^{I}$, on pose:
\begin{center}
$X_{I,w}=G_{ad}.[I,w,1]$
\end{center}
où $[I,w,1]:=(B_{ad}\times B_{ad}).w h_{I}$ où l'action de $B_{ad}\times B_{ad}$ est donnée par la multiplication à gauche et à droite et $G_{ad}$ agit par conjugaison.

\begin{thm}\label{he1}[Lusztig-He]
\begin{enumerate}
\item
$X_{I}$ est l'union disjointe des $X_{I,w}$, $w\in W^{I}$.
\item
$X_{I,w}$ est localement fermé et irréductible de dimension $\dime G -l(w)-\left|\Delta-I\right|$.
\end{enumerate}
\end{thm}
La preuve de (i) est donnée par Lusztig \cite{L} et He \cite{H1} et (ii) est montré par Lusztig \cite[sect. 8]{L2}.
Suivant \cite{H2}, nous avons les relations suivantes entre les adhérences des strates $X_{I,w}$.

Soit $\mathcal{I}$ l'ensemble des paires $(I,x)$ avec $I\subset \Delta, x\in W^{I}$. On définit la relation $\leq$ sur $\mathcal{I}$ par:
\begin{center}
$(I,x)\leq (K,y)$ si et seulement si $I\subset K$ et $x\geq z^{-1}yz$ pour $z\in W_{K}$.
\end{center}
Nous avons alors le théorème suivant dû à He sur l'adhérence des strates:
\begin{thm}\label{he2}
\begin{enumerate}
\item
La relation $\leq$ définit une relation d'ordre sur $\mathcal{I}$.
\item
Si $(I,x)$, $(K,y) \in \mathcal{I}$ alors $X_{I,x}\subset \overline{X_{K,y}}$ si et seulement si $(I,x)\leq (K,y)$.
\end{enumerate}
\end{thm}
On se reporte à \cite[sect. 3, 4]{H2}  pour la preuve. 
On peut maintenant faire le lien avec le cône nilpotent.
Pour $w\in W$, on définit le support de $w$, $\supp(w)\subset\Delta$ comme l'ensemble des réflexions simples qui interviennent dans une décomposition réduite de $w$. Soit 
\begin{center}
$\mathcal{N}_{ad}:=\{x\in X\vert~ \forall~ i, \rho_{\omega_{i}}(x)\in \mathbb{P}(\End(V_{\omega_{i}}))$  \text{est nilpotent}\}.
\end{center}

\begin{thm}(He)\label{hehe}
Nous avons la stratification suivante de $\mathcal{N}_{ad}$:
\begin{center}
$\mathcal{N}_{ad}=\coprod\limits_{I\neq\Delta}\coprod\limits_{\substack{w\in W^{I}
            \\\supp(w)=\Delta}}X_{I,w}$.
\end{center}
\end{thm}
Ce théorème fut d'abord démontré par He dans \cite[Thms. 4.3, 4.5]{H1}  au cas par cas. Springer, dans \cite[sect. 3.3]{Sp1}, donne une preuve plus courte.
Revenons à notre problème initial, à savoir l'étude de la fibre $\chi_{+}^{-1}(0)$.
On rappelle que nous avons une flèche :
\begin{center}
$\sigma: V_{G}^{0}\rightarrow X$ 
\end{center}
qui est un $Z_{+}$-torseur et la fibre $\mathcal{N}^{0}=\chi_{+}^{-1}(0)\cap V_{G}^{0}$ s'envoie donc surjectivement dans la strate $\mathcal{N}_{ad,\emptyset}$. On a le diagramme cartésien suivant:
$$\xymatrix{\mathcal{N}^{0}\ar[d]\ar[r]&V_{G}^{0}\ar[d]\\\mathcal{N}_{ad,\emptyset}\ar[r]&X}$$
On rappelle que le fait d'être dans $\chi_{+}^{-1}(0)$ impose de tomber dans la strate vide.
On note de la même manière l'image réciproque de $X_{\emptyset,w}$ dans le semi-groupe de Vinberg.
On en déduit la proposition suivante:
\begin{prop}
Le cône nilpotent $\mathcal{N}^{0}$ est équidimensionnel de dimension $\dime G_{+}-2r=\dime G-r$. Ses composantes irréductibles sont indexées par les éléments de Coxeter de $W$.
\end{prop}
\begin{proof}
Il résulte du diagramme cartésien ci-dessus que $\dim\mathcal{N}^{0}=\dime \mathcal{N}_{ad,\emptyset}+r$.
Comme nous avons d'après le théorème \ref{hehe}:
\begin{center}
$\mathcal{N}_{ad,\emptyset}=\coprod\limits_{\substack{w\in W\\\supp(w)=\Delta}}X_{\emptyset,w}$,
\end{center}
il résulte du théorème \ref{he1}.(ii) que la dimension $\mathcal{N}_{ad,\emptyset}$ est égale à la dimension de $X_{\emptyset,w}$ lorsque $w$ est un élément de Coxeter, autrement dit $l(w)=r$ et $\dim(X_{\emptyset,w})=\dim G-2r$.
Le calcul de dimension suit.
Il ne nous reste plus qu'à voir l'assertion sur les composantes irréductibles. Elle résulte du théorème \ref{he1}.(ii) et de la proposition suivante due à He \cite[Prop 2.10]{H1} :
\begin{prop}\label{he3}
Soit $w\in W$ tel que $\supp(w)=\Delta$, alors il existe un élément de Coxeter $w'\in W$ tel que $w'\leq w$.
\end{prop}
\end{proof}
Nous allons maintenant caractériser les éléments réguliers.
\begin{prop}\label{nilreg}
On a la décomposition suivante de $\mathcal{N}^{reg}$:
\begin{center}
$\mathcal{N}^{reg}=\coprod\limits_{\substack{w\in W
            \\\supp(w)=\Delta\\ l(w)=r}}X_{\emptyset,w}.$
\end{center}
D'où l'on déduit que chaque strate $X_{\emptyset,w}$ est une classe de conjugaison.
Enfin, $\mathcal{N}^{reg}$ est dense dans $\mathcal{N}^{0}$ et $\mathcal{N}^{reg}$ est lisse.
\end{prop}
$\rmqs$ 
\begin{itemize}
\item
La classe de conjugaison de la section de Steinberg correspond à la composante connexe associée à l'élément $w=s_{1}\dots s_{r}$.
On voit alors qu'en prenant un autre élément de Coxeter, on obtient une section qui tombe dans une composante distincte.
Cela n'apparaît pas au niveau du groupe des inversibles, vu que le choix d'un autre élément de Coxeter aboutit à une section conjuguée, ce qui n'est pas le cas si l'on considère le semi-groupe, à cause de l'apparition des idempotents.
\item
On obtient donc une famille $(\epsilon_{+}^{w'})$ de sections indexées par les éléments de Coxeter $w'\in W$ et la proposition ci-dessus nous dit que les éléments nilpotents réguliers sont conjugués à une certaine section $\epsilon_{+}^{w'}$.
\end{itemize}
\begin{proof}
Pour prouver la proposition, il suffit donc de calculer la dimension du centralisateur des éléments $w'e_{\emptyset}$, pour $w'$ un élément de Coxeter. Le calcul est alors le même que pour le lemme \ref{cent0}. On déduit ensuite, du fait que $X_{\emptyset, w'}$ est irréductible de bonne dimension, qu'il s'identifie à la classe  de conjugaison de $w'e_{\emptyset}$.
La densité vient de la proposition \ref{he3}.
  
Pour la lissité, on regarde la flèche $\chi_{+}^{reg}:V_{G}^{reg}\rightarrow\mathfrak{C}_{+}$. C'est une flèche entre schémas lisses. Comme $\epsilon_{+}$ est une section, on sait que l'application tangente en $\epsilon_{+}(0)$ est surjective et donc il en est de même de tous ses conjugués.

Pour ce qui est  des autres composantes connexes, à chaque élément de Coxeter $w'$, on peut définir de même une section $\epsilon_{+}^{w'}$ à $\chi_{+}$ et on obtient de même la lissité.
\end{proof}
On termine notre étude du cône nilpotent en étendant les résultats de $\mathcal{N}^{0}$ au cône nilpotent $\mathcal{N}$.
On commence par  définir pour $J\subset\Delta$ et $w\in W$, les strates $X_{\emptyset, J, w}$:
\begin{center}
$X_{\emptyset, J, w}:=G.(Bwe_{\emptyset, J}B)$
\end{center}
où les $e_{\emptyset, J}$ correspondent aux idempotents du semigroupe de Vinberg dans les strates plus petites. On remarque que si $J=\Delta$, on est dans $\mathcal{N}^{0}$.

On rappelle que les idempotents sont nuls dans les représentations $V_{\omega_{i}}$ pour $i\notin J$ et qu'ils correspondent à la projection sur le vecteur de plus haut poids $v_{\omega_{j}}$ dans $\End(V_{\omega_{j}})$ pour $j\in J$.
On en déduit alors que, pour tout $J\in\Delta$ et $w\in W$ avec $J\subset \supp(w)$, $X_{\emptyset, J, w}\subset\mathcal{N}$. La proposition suivante montre la réciproque.

\begin{prop}\label{nilstrat}
Le cône nilpotent admet la stratification suivante:
\begin{center}
$\mathcal{N}=\coprod\limits_{J\subset\Delta}\coprod\limits_{\substack{w\in W
            \\\supp(w)=\Delta}}X_{\emptyset, J, w}$.
\end{center}
En particulier, $\mathcal{N}^{0}$ est dense dans $\mathcal{N}$ et $\dim\mathcal{N}= \dime G_{+}-2r$.
\end{prop}

\begin{proof}
Il nous suffit de démontrer la décomposition, vu que le reste est un corollaire immédiat de la description des strates.
Soit  $x\in \mathcal{N}$, il est dans une certaine strate $\mathcal{O}_{\emptyset, J}$.
Quitte à conjuguer, on peut supposer qu'il est de la forme $ge_{\emptyset, J}$.
On écrit que $g$ est dans une certaine orbite $BwB$, pour $w\in W$. 
Comme  le radical unipotent $U$ fixe $e_{\emptyset, J}$, on peut supposer que $x=bwe_{\emptyset, J}$.
On commence par voir que $J\subset \supp(w)$. On sait déjà que $\rho_{\omega_{i}} e_{\emptyset, J}$ est nul pour $i\notin J$.

Soit alors $j\in J$, on considère son image dans $\End(V_{\omega_{j}})$.
Nous avons $e_{\emptyset, J}v_{\lambda}=0$ pour $\lambda\neq\omega_{j}$. Etudions la contribution de $v_{\omega_{j}}$.
Soit une décomposition réduite de $w=s_{i_{1}}s_{i_{2}}...s_{i_{l}}$.
Si par l'absurde la réflexion simple  $s_{j}$, n'intervenait pas dans cette décomposition, le coefficient devant $v_{\omega_{j}}$ serait non nul ce qui est contradictoire avec la nilpotence.

On obtient donc $J\subset \supp(w)$. Si $\supp(w)=\Delta$, on a rien à montrer et si $\supp(w)$ est un sous ensemble strict de $\Delta$, on considère l'élément $w'=ws_{i_{1}}s_{i_{2}}...s_{i_{p}}$ avec pour tout $l\leq p$, $i_{l}\in\Delta-\supp(w)\subset \Delta-J$.
De la description de $e_{\emptyset, J}$, on a que $we_{\emptyset, J}=w'e_{\emptyset, J}$, ce qui termine la preuve de la proposition.
\end{proof}

\subsection{La lissité du  morphisme $\chi_{+}^{reg}$}
\begin{prop}
Soit $\chi_{+}:V_{G}\rightarrow\mathfrak{C}_{+}$. Alors, $\chi_{+}$ est plat et ses fibres sont de dimension $\dime G_{+}-2r$.
\end{prop}
\begin{proof}
Au point $0$, on a vu que le cône nilpotent était de la bonne dimension. On considère alors l'ensemble:
\begin{center}
$U:=\{a\in \mathfrak{C}_{+}\vert~\dime\chi_{+}^{-1}(a)=\dime G_{+}-2r\}$.
\end{center}
Par EGA IV 9.3.2, cet ensemble est constructible et par EGA IV 13.1.1,\\
comme $\dime G_{+}-2r$ est la dimension de la fibre générique de $\chi_{+}$, $U$ est stable par générisation, donc c'est un ouvert. L'argument standard utilisé ci-dessus, nous montre que $U=\mathfrak{C}_{+}$. Enfin, d'après Brion-Kumar \cite[6.2.9 et 6.2.11]{BK}, $V_{G}$ est de Cohen-Macaulay. Ainsi, la base étant lisse et les fibres étant toutes de même dimension, on en déduit que le morphisme est plat.
\end{proof}
\begin{prop}\label{stcoxeter}
Le morphisme $\chi_{+}^{reg}:V_{G}^{reg}\rightarrow\mathfrak{C}_{+}$ est lisse. 
La flèche $\chi_{+}:V_{G}\rightarrow\mathfrak{C}_{+}$ est à fibres géométriquement réduites et Cohen-Macaulay et les éléments réguliers forment un ouvert dense dans chaque fibre.
\end{prop}
\begin{proof}
Pour chaque élément $w\in W$ de Coxeter, on regarde la flèche 
\begin{center}
$\pi^{w}:G\times\mathfrak{C_{+}}\rightarrow V_{G}^{reg}$
\end{center}
 donnée par 
\begin{center}
$(g,a)\rightarrow g\epsilon_{+}^{w}(a)g^{-1}$.
\end{center}
Cette flèche admet la factorisation suivante:
$$\xymatrix{G\times\mathfrak{C}_{+}\ar[dr]^{\rho}\ar[dd]_{\pi^{w}}\\&(G\times\mathfrak{C}_{+})/J\ar[dl]^{i}\\V_{G}^{reg}}$$

La flèche $\rho$ est un $J$-torseur et $i$ est quasi-fini. Par Steinberg \cite[Th. 8.1]{S} , on sait que $i$ est un isomorphisme au-dessus de $G_{+}^{reg}$, on en déduit que $i$ est birationnel et quasi fini et comme $V_{G}^{reg}$ est normal, par le Main Theorem de Zariski, c'est une immersion ouverte. On note alors $V^{w}$ l'ouvert image.
Maintenant, on pose:
\begin{center}
$U=\bigcup\limits_{\substack{w\in W\\\supp(w)=\Delta\\ l(w)=r}} V^{w}$
\end{center}
et soit $F$ son fermé complémentaire dans $V_{G}$. Posons $l_{G}:= \dime G_{+}-2r$.

Soit $V:=\{x\in F\vert~\dime_{x}\chi_{+}^{-1}(\chi_{+}(x))\leq l_{G}-1\}$. Par le théorème de Chevalley sur la semi-continuité des fibres, $V$ est ouvert dans $F$ et par la proposition \ref{nilreg}, il contient 0 et est $Z_{+}$-équivariant; il est donc égal à $F$. On en déduit que $U$ est de codimension au moins un dans chaque fibre et comme les fibres sont équidimensionnelles, $U$ est dense dans chaque fibre.
En particulier, tout élément régulier est conjugué à une section $\epsilon_{w}$, pour un certain élément de Coxeter $w$.
Et donc le même argument que pour la proposition \ref{nilreg}, nous dit qu'en tout élément régulier $x$, la différentielle est  surjective, d'où la lissité, puisque la source et le but sont lisses.

Enfin, les fibres de $\chi_{+}$ sont Cohen-Macaulay, car on a une immersion régulière de la fibre dans $V_{G}$, qui est Cohen-Macaulay. Comme les éléments réguliers forment un ouvert lisse dense dans un schéma Cohen-Macaulay, les fibres sont réduites.
\end{proof}
\begin{prop}
Le centralisateur régulier $J$ est un schéma en groupes lisse et commutatif de dimension $r$.
\end{prop}
\begin{proof}
On a le diagramme cartésien:
$$\xymatrix{I_{\vert V_{G}^{reg}}\ar[r]\ar[d]&G\times V_{G}^{reg}\ar[d]^{\psi}\\V_{G}^{reg}\ar[r]&V_{G}^{reg}\times_{\mathfrak{C}_{+}} V_{G}^{reg}}$$
où la flèche $\psi$ est donnée par $(x,y)\rightarrow (xyx^{-1},y)$ et la flèche du bas par la diagonale. Comme elle est équidimensionnelle entre schémas lisses, elle est plate.
Ainsi, on déduit que la flèche $i:I_{\vert V_{G}^{reg}}\rightarrow G\times V_{G}^{reg}$ est une immersion régulière et donc $I_{\vert V_{G}^{reg}}$ est intersection complète. Comme de plus, au-dessus de $V_{G}^{reg}$, $I$ est équidimensionnel sur une base lisse, il est plat et par changement de base $J$ est également plat sur $\mathfrak{C}_{+}$.

Soit $\phi$ le morphisme structural de $J$. Au-dessus de 0, d'après le lemme \ref{cent0}, la fibre  est lisse. D'après la remarque \ref{principe}, il nous suffit de montrer que l'ensemble des $a\in\mathfrak{C}_{+}$ tels que la fibre de $J_{a}$ est lisse est ouvert. Comme $J$ est plat sur $\mathfrak{C}_{+}$, on a d'après EGA IV 12.2, que l'ensemble des $x\in J$, tel que $J_{\phi(x)}$ est lisse en $x$ est ouvert.

Or, de par sa structure de schéma en groupes, si la fibre est lisse en un point, elle est lisse partout, donc on en déduit, toujours comme $J$ est plat sur sa base que, $\{a\in\mathfrak{C}_{+}\vert~ J_{a}$ est lisse\} est ouvert. Comme il est clairement $Z_{+}$-équivariant, la remarque du lemme \ref{cent0}  s'applique.

Enfin, comme au-dessus de $G_{+}^{reg}$, le centralisateur régulier est abélien, on a donc un schéma en groupes lisse, qui est génériquement abélien, donc abélien.
\end{proof}
Nous terminons par un critère analogue à Steinberg \cite[3.2 et 3.3]{S}  pour caractériser les réguliers nilpotents.
\begin{prop}\label{stnil}
Tout élément régulier nilpotent de $V_{G}$ est dans un unique semi-groupe de Borel.
\end{prop}
\begin{proof}
Il nous suffit de le montrer pour chaque nilpotent régulier dans une composante connexe fixée. On prend celle qui correspond à l'élément de Coxeter $w=s_{1}\dots s_{r}$, la preuve étant la même pour tout autre élément dans une autre composante connexe. Quitte à conjuguer, on a juste à montrer le résultat pour $\epsilon_{+}(0)$.
Soit alors $\epsilon_{+}(0)=we_{\emptyset}$ avec $w=s_{1}\dots s_{r}$, qui serait dans deux semi-groupes de Borel $V_{B'}$ et $V_{B^{-}}$, alors $B^{-}$ et $B'$ étant conjugués, en écrivant la décomposition de Bruhat relativement à $B^{-}$, nous avons $B'=u\sigma B^{-}\sigma^{-1}u^{-1}$, avec $\sigma\in W$ et $u\in U^{-}$.

Dans ce cas, on obtient que l'élément $\sigma^{-1}u^{-1}we_{\emptyset}u\sigma$ est dans $V_{B}^{-}$.
On se fixe alors $i$ et on regarde l'action sur les vecteurs $v_{\lambda}$:
\begin{center}
$\sigma^{-1}u^{-1}we_{\emptyset}u\sigma v_{\lambda}=\sigma^{-1}u^{-1}we_{\emptyset}uv_{\sigma(\lambda)}$.
\end{center}
En particulier, comme $u\in U^{-}$ et par définition de $e_{\emptyset}$, on en déduit que si $\sigma(\lambda)\neq\omega_{i}$, ce vecteur est nul. Etudions le cas résiduel $\sigma(\lambda)=\omega_{i}$, on a alors:
\begin{center}
$\sigma^{-1}u^{-1}we_{\emptyset}u\sigma v_{\lambda}=\sigma^{-1}u^{-1}[v_{\omega_{i}-\alpha_{i}-\beta}]$
\end{center}
où $\beta$ est une somme à coefficients positifs de racines simples $\alpha_{j}$ pour $j\leq i-1$. D'où l'on déduit:
\begin{center}
$\sigma^{-1}u^{-1}we_{\emptyset}u\sigma v_{\lambda}=v_{\sigma^{-1}(\omega_{i}-\alpha_{i}-\beta)}$+(autres vecteurs).
\end{center}
Or $\sigma^{-1}(\omega_{i}-\alpha_{i}-\beta)=\lambda-\sigma^{-1}\alpha_{i}-\sigma^{-1}\beta$, et comme l'élément doit être dans $V_{B}^{-}$, cela force 
\begin{center}
$\lambda-\sigma^{-1}\alpha_{i}-\sigma^{-1}\beta\leq \lambda$.
\end{center}
Cette inégalité étant valable pour tout $i$, cela donne $\sigma=1$, ce qu'on voulait.
\end{proof}

Considérons le schéma $\tilde{V}_{G}:=\{(g,\g)\in G/B\times V_{G}\vert~ g^{-1}\g g\in V_{B}\}$ qui est fermé dans $G/B\times V_{G}$.
On a un morphisme propre $\la:\tilde{V}_{G}\rightarrow V_{G}$ car il admet la factorisation suivante:
$$\xymatrix{\tilde{V}_{G}\ar[r]\ar[dr]^{\la}&V_{G}\times G/B\ar[d]\\&V_{G}}$$

Nous allons nous intéresser à ce morphisme au-dessus du lieu régulier:
\begin{cor}\label{finspring}
Le morphisme $\la^{reg}:\tilde{V}_{G}^{reg}\rightarrow V_{G}^{reg}$ est fini plat.

\end{cor}

\begin{proof}
Comme on sait déjà que $\la^{reg}$ est propre par changement de base, il nous suffit de vérifier, en vertu du Main Theorem de Zariski, qu'elle est quasi-finie.
Il résulte alors de la proposition \ref{stnil} qu'au-dessus de $0\in\kc$, la fibre est réduite à un point, donc par semi-continuité de la dimension des fibres, on en déduit, que le morphisme est partout quasi-fini.

Passons à la platitude, considérons le diagramme commutatif suivant:
$$\xymatrix{\tilde{V}_{G}^{reg}\ar[r]\ar[d]_{\la}&V_{T}\ar[d]^{\theta}\\V_{G}^{reg}\ar[r]&\kc}$$
Ce diagramme est en fait cartésien. En effet, la flèche $\chi_{+}^{reg}$ est lisse, donc $V_{G}^{reg}\times_{\kc}V_{T}$ est lisse au-dessus de $V_{T}$, donc normal et Cohen-Macaulay.
La flèche $\iota:\tilde{V}_{G}^{reg}\rightarrow V_{G}^{reg}\times_{\kc}V_{T}$ est finie car $\tilde{V}_{G}^{reg}\rightarrow V_{G}^{reg}$ l'est. De plus, $\iota$ est birationnelle car c'est un isomorphisme au-dessus de $G_{+}^{rs}$, il résulte alors du Main Theorem de Zariski que $\iota$ est un isomorphisme.

On en déduit alors que $\tilde{V}_{G}^{reg}$ est également Cohen-Macaulay et fini surjectif sur un schéma lisse, donc plat.
\end{proof}
Nous avons également la réciproque à la proposition \ref{stnil}:
\begin{prop}\label{stnilb}
Soit $x\in V_{G}$ un élément nilpotent irrégulier, alors il est dans une infinité de semi-groupes de Borel.
\end{prop}
\begin{proof}
Il nous suffit de montrer l'assertion pour $x\in V_{G}^{0}$. On rappelle que nous avons la stratification suivante pour $\mathcal{N}^{0}$:
\begin{center}
$\mathcal{N}^{0}=\coprod\limits_{\substack{w\in W \\\supp(w)=\Delta}} X_{\emptyset,w}$
\end{center}
et que  d'après la proposition $\ref{nilreg}$, les éléments réguliers correspondent aux éléments de longueur $r$.
En particulier, si $x\in\mathcal{N}^{0}$ est irrégulier,  il est dans une strate $X_{\emptyset, w}$ pour $l(w)\geq r+1$.
On commence donc par démontrer la proposition pour $x\in X_{\emptyset,w}$ avec $l(w)=r+1$.
Pour démontrer le résultat, quitte à conjuguer, on peut supposer que:
\begin{center}
$x=nwe_{\emptyset}$
\end{center}
avec $n\in U$ et $w\in W$ tel que $l(w)=r+1$ et $\supp(w)=\Delta$.
L'assertion va résulter du calcul du centralisateur et du fait que l'élément considéré va être  dans un certain Lévi.
Regardons le cas de l'élément $we_{\emptyset}$, le cas d'un élément de la forme $nwe_{\emptyset}$ étant analogue.

En reprenant le calcul du lemme $\ref{cent0}$, on obtient que le centralisateur de $I_{we_{\emptyset}}$ s'identifie alors à $Z_{G}T_{\beta}U_{w}^{-}$ où $U_{w}^{-}=U^{-}\cap wU w^{-1}$ et $T_{\beta}:=\{t\in T\vert~wtw^{-1}=t\}$ est un tore de dimension 1 engendré par une certaine coracine $\check{\beta}$ avec $w\beta=\beta$.
Soit alors $L_{\beta}:=Z_{G_{+}}(T_{\beta})$, alors $we_{\emptyset}$ est dans  $L_{\beta}$.

Et de l'égalité $w\beta=\beta$, on en déduit que pour tout $u\in U_{-\beta}$:
\begin{center}
$we_{\emptyset}\in u\sigma_{\beta}V_{B}^{-}\sigma_{\beta}^{-1}u^{-1}$.
\end{center}
Ainsi, $we_{\emptyset}$ est dans une infinité de semi-groupes de Borel.
Il nous suffit désormais de montrer que l'adhérence des strates $X_{\emptyset,w}$ pour $l(w)=r+1$ consiste en tous les éléments nilpotents irréguliers.
En effet, d'après le théorème de Chevalley, la dimension des fibres de $\la:\tilde{V}_{G}\rightarrow V_{G}$ ne peut qu'augmenter, donc on aura le résultat pour les autres éléments.
En vertu de la relation d'adhérence entre les strates (cf. Thm. \ref{he2}), il nous suffit de démontrer le lemme suivant:
\begin{lem}(Premet)
Soit $w\in W$ tel que $\supp(w)=\Delta$ et $l(w)\geq r+2$, alors il existe $w'\in W$ de même support que $w$ et tel que $l(w')=r+1$.
\end{lem}
La preuve qui suit est due à Premet:
\medskip

Posons $l=l(w)$ et considérons $\Red(w)$ l'ensemble de toutes les expressions réduites de $w$. Etant donné $\textbf{r}=(i_{1},\dots, i_{l})\in \Red(w)$, on note $k(\textbf{r})$ le plus petit $k\leq l$ tel que $i_{k}$ apparaît dans $\textbf{r}$ plus de deux fois.
Soit $\textbf{t}=(j_{1},\dots, j_{l})\in \Red(w)$ tel que $k(\textbf{t})\geq k(\textbf{r})$ pour tout $\textbf{r}\in \Red(w)$.
On écrit alors $w=s_{j_{1}}\dots s_{j_{l}}$ et on considère $w'\in W$ obtenu à partir de $w$ en supprimant $s_{j_{k}}$ de l'expression réduite de $w$ où
$k=k(\textbf{t})$.
De par le choix de $k$, $w'$ a le même support que $w$ et sa longueur a diminué de $1$. En itérant le procédé, on fait diminuer la longueur jusqu'à $r+1$ ce qu'on voulait.
\end{proof}

On termine le paragraphe par une étude du discriminant sur le semi-groupe de Vinberg.
On considère la fonction $\mathfrak{D}_{+}=(2\rho,\mathfrak{D})$ sur $k[T_{+}]$ où $\mathfrak{D}=\prod\limits_{\alpha\in R}(1-\alpha(t))$ est la fonction discriminant sur $k[T]$.
Comme $W$ agit trivialement sur le premier facteur, on a que $\mathfrak{D}_{+}\in k[T_{+}]^{W}$ et de plus on a la propriété suivante:
\begin{center}
$t\in T_{+}^{rs}\Longleftrightarrow \mathfrak{D}_{+}(t)\neq 0$.
\end{center}
Nous allons étendre cette fonction à $V_{T}$:
\begin{lem}\label{prongdisc}
La fonction $\mathfrak{D}_{+}$ s'étend en une fonction de $k[V_{T}]^{W}$. 
\end{lem}
\begin{proof}
On commence par étendre la fonction à $Z_{+}\times \overline{T}_{\Delta}$.
Soit $t_{+}=(t,t^{-1})\in T_{\Delta}$, alors nous avons:
\begin{center}
$\mathfrak{D}_{+}(t_{+})=2\rho(t)\prod\limits_{\alpha>0}(1-\alpha(t^{-1}))(1-\alpha(t))$
\end{center}
Comme $2\rho=\sum\limits_{\alpha>0}\alpha$, nous obtenons:
\begin{equation}
\mathfrak{D}_{+}(t_{+})=(-1)^{\left|R^{+}\right|}\prod\limits_{\alpha>0}(1-\alpha(t))^{2}
\label{extdisc}
\end{equation}
où $R^{+}$ est l'ensemble des racines positives.
En particulier, on obtient que $\mathfrak{D}_{+}$ se prolonge à $Z_{+}\times \overline{T}_{\Delta}$. Maintenant, comme $\mathfrak{D}_{+}$ est $W$-invariante, on obtient dans un premier temps, que $\mathfrak{D}_{+}$ s'étend à $V_{T}^{0}$, l'adhérence de $T_{+}$ dans $V_{G}^{0}$.
Maintenant, comme $V_{T}$ est affine et normal et que la codimension du complémentaire de $V_{T}^{0}$ dans $V_{T}$ est au moins deux, on obtient que $\mathfrak{D}_{+}$ se prolonge.
\end{proof}

Il ne nous reste plus qu'à voir le critère de régularité.
\begin{prop}\label{cardisc}
Soit $t_{+}\in V_{T}$, alors nous avons l'équivalence:
\begin{center}
$t_{+}\in V_{T}^{rs}\Longleftrightarrow \mathfrak{D}_{+}(t_{+})\neq 0$.
\end{center}
\end{prop}
\begin{proof}
Soit $t_{+}\in V_{T}^{rs}$, en particulier $t_{+}\in V_{T}\cap V_{G}^{0}=V_{T}^{0}$. Quitte à conjuguer par un élément de $W$, on peut supposer que $t\in Z_{+}\overline{T}_{\Delta}$.
Si nous avons $\mathfrak{D}_{+}(t)=0$, alors il résulte de l'égalité \eqref{extdisc} qu'il existe $\alpha\in R$ tel que $\alpha(t_{+})=1$.
En particulier, en considérant le tore $T_{\alpha}=\Ker \alpha\subset T_{+}$, son centralisateur dans $G$ est de dimension strictement plus grande que $r$ et donc $t_{+}\in\overline{T}_{\alpha}$ ne peut être régulier semisimple.

Soit alors $F$ le fermé complémentaire de $V_{T}^{rs}$ dans $V_{T}$. On sait que $F$ est un diviseur, comme c'est le complémentaire du lieu étale de la flèche $V_{T}\rightarrow\kc$, en particulier équidimensionnel.
Soit $U\subset V_{T}$ le lieu où $\mathfrak{D}_{+}$ ne s'annule pas et $F_{1}$ son fermé complémentaire.
On vient de voir que $V_{T}^{rs}\subset U$.

Supposons par l'absurde que $V_{T}^{rs}$ soit strictement inclus dans $U$ et considérons le fermé $F_{2}=U-V_{T}^{rs}$.
Nous avons alors $F_{1}\cap F_{2}=\emptyset$ et $F=F_{1}\cup F_{2}$.
Lorsque l'on intersecte $V_{T}^{rs}$ et $U$ avec $T_{+}$, les ouverts sont les mêmes.

De plus, il résulte de la proposition \ref{stabet} que l'élément $e_{\emptyset}$ est régulier semisimple (si $ge_{\emptyset} g^{-1}=e_{\emptyset}$, alors $g\in B\cap B^{-}$).
Pour $1\leq i\leq r$, si $\co_{i}$ est la strate de codimension un de $T_{+}$, nous avons que:
\begin{center}
$e_{\emptyset}\in\bigcap\limits_{i=1}^{r}\co_{i}$. 
\end{center}
En particulier, pour tout $1\leq i\leq r$, $V_{T}^{rs}\cap\co_{i}$ est un ouvert non vide de $\co_{i}$. En particulier, les points génériques des strates $\co_{i}$ sont dans $V_{T}^{rs}$ d'où l'on déduit que $F_{2}$ est de codimension au moins deux.

Or, comme $F$ est équidimensionnel, on obtient que $F_{2}=\emptyset$, une contradiction.
\end{proof}

\subsection{Une construction alternative du centralisateur}\label{galoiscent}

Dans ce paragraphe nous donnons une interprétation alternative du centralisateur régulier d'après Donagi-Gaitsgory \cite{DG} et Ngô \cite[sect. 2.4]{N}.
Toutefois, il y a une subtile différence avec le cas de Ngô et de Donagi-Gaitsgory, dans la mesure où les fibres du morphisme de Steinberg ne sont pas intègres. En particulier, on voit qu'au point 0, les sections de Steinberg pour deux éléments de Coxeter donnent des éléments réguliers non conjugués. Néanmoins, nous allons voir que les centralisateurs restent isomorphes.
Nous supposons de plus que la caractéristique du corps est première à l'ordre du groupe de Weyl.
Dans la preuve de la proposition \ref{stcoxeter}, pour chaque élément de Coxeter $w\in W$ nous  avions introduit la flèche:
\begin{center}
$G\times\kc\rightarrow V_{G}^{reg}$
\end{center}
définie par $(g,a)\mapsto g\eps_{+}^{w}(a)g^{-1}$ dont nous avions vu qu'elle était d'image ouverte.
On note $V_{G}^{reg,w}\subset V_{G}^{reg}$ l'ouvert image. En particulier, pour chaque élément de Coxeter $w\in W$, nous avons un centralisateur régulier $J_{w}$ lisse et de dimension $r$ qui sont a priori différents deux à deux. On commence par la proposition suivante:
\begin{prop}\label{proJ}
Il existe un unique schéma en groupes $J$, lisse et commutatif de dimension $r$ sur $\mathfrak{C}_{+}$ muni d'un isomorphisme $G_{+}$-équivariant:
\begin{center}
$\chi_{+}^{*}J_{V_{G}^{reg}}\stackrel{\cong}{\rightarrow} I_{V_{G}^{reg}}$.
\end{center}
De plus, cet isomorphisme se prolonge en une flèche de $\chi_{+}^{*}J\rightarrow I$.
\end{prop}

\begin{proof}
Pour chaque élément de Coxeter $w\in W$, en reprenant exactement la preuve par descente fidèlement plate  de \cite[Lem. 2.1.1]{N}, nous obtenons un unique schéma en groupe commutatif lisse $J_{w}$ qui est isomorphe à $I_{V_{G}^{reg,w}}$.
Soient donc deux éléments de Coxeter distincts $w,w'\in W$, nous allons montrer que les schémas en groupes lisses $J_{w}$ et $J_{w'}$ sont isomorphes. Comme ils sont lisses, il suffit de montrer l'isomorphisme sur un ouvert dont le complémentaire est de codimension au moins deux de $\kc$.
Cela résulte alors de la proposition suivante:
\begin{lem}
Pour tout élément de  Coxeter, nous avons $G_{+}^{reg}\cup V_{G}^{rs}\subset V_{G}^{reg,w}$, où $V_{G}^{rs}$ est le lieu où le centralisateur est un tore.
\end{lem}
\begin{proof}
D'après Steinberg, au-dessus de $G_{+}^{reg}$, on sait que tout élément est conjugué à $\eps_{+}(\kc)$ et de plus au-dessus de $V_{G}^{rs}$ le centralisateur étant un tore et tous les tores étant conjugués, on obtient que tout élément de $V_{G}^{rs}$ est également conjugué à un élément de $\eps_{+}(\kc^{rs})$, en particulier $G_{+}^{reg}\cup V_{G}^{rs}\subset V_{G}^{reg,w}$.
\end{proof}
Comme pour tout $1\leq i\leq r$, nous avons vu dans la preuve de la proposition \ref{cardisc} que $V_{G}^{rs}\cap\mathcal{O}_{i}$ était de codimension un dans $\mathcal{O}_{i}$, nous obtenons que 
la codimension du fermé complémentaire de $G_{+}^{reg}\cup V_{G}^{rs}$ est déjà de codimension deux, donc également son image dans $\kc$ par lissité de la flèche $V_{G}^{reg}\rightarrow\kc$ et donc nous obtenons que $J_{w}$ et $J_{w'}$ sont naturellement isomorphes.

Nous déduisons donc qu'il existe un unique schéma en groupes lisse $J$ sur $\kc$ tel que:
\begin{center}
$\chi_{+}^{*}J_{V_{G}^{reg}}\stackrel{\cong}{\rightarrow} I_{V_{G}^{reg}}$.
\end{center}
Il ne reste donc qu'à montrer que l'isomorphisme se prolonge en une flèche de $\chi_{+}^{*}J\rightarrow I$.
Comme la codimension du fermé complémentaire de $V_{G}^{reg}$ dans $V_{G}$ est de codimension au moins deux, il s'ensuit que la flèche se prolonge en un morphisme :
\begin{center}
$\chi_{+}^{*}J\rightarrow I$
\end{center}
au-dessus de $V_{G}$.
\end{proof}

\begin{prop}
$[\chi_{+}]:[V_{G}^{reg}/G]\rightarrow\mathfrak{C}_{+}$ est une gerbe liée par le centralisateur $J$ et neutre.
\end{prop}
\begin{proof}
Nous avons vu que le morphisme était lisse. De par la caractérisation du centralisateur régulier, le faisceau des automorphismes d'un élément $(E,\phi_{+})\in [V_{G}^{reg}/G](S)$ au-dessus d'un point $a:S\rightarrow\mathfrak{C}_{+}$ est canoniquement isomorphe à $a^{*}J$. La neutralité vient de l'existence de $\epsilon_{+}$.
\end{proof}

\section{Les fibres de Springer affines et leur interprétation modulaire}
\subsection{Grassmannienne affine}

Soit $F=k((\pi))$ un corps local, d'anneau d'entiers $\mathcal{O}$ et de corps résiduel $k$ algébriquement clos. Soit $G$  connexe semisimple simplement connexe. Soit $K=G(\co)$.
On note $X_{*}(T)^{+}$ l'ensemble des caractères dominants.
Sur $G(F)$, on a la décomposition de Cartan:
\begin{center}
$G(F)=\coprod\limits_{\lambda\in X_{*}(T)^{+}}K\pi^{\lambda}K$ 
\end{center} 
où  $\lambda$ un cocaractère dominant.
Soit $\overline{K\pi^{\la}K}$ l'adhérence de $K\pi^{\lambda}K$ dans $G(F)$, nous avons la stratification suivante:
\begin{center}
$\overline{K\pi^{\la}K}=\coprod\limits_{\mu\leq\la}K\pi^{\mu}K$.
\end{center}
Soit $\rho_{\omega_{i}}$ la représentation irréductible de plus haut poids $\omega_{i}$. 
Dans la section \ref{rapsemi}, nous avons introduit le semi-groupe de Vinberg $V_{G}$ ainsi que son ouvert lisse $V_{G}^{0}$. Le lien entre le semi-groupe de Vinberg et la grassmannienne affine apparaît ici.
\begin{lem}\label{cont} 
Un élément $g\in G(F)$ appartient à l'orbite $K\pi^{\lambda}K$ (resp. $\overline{K\pi^{\la}K}$) si et seulement si pour tout cocaractère dominant $\omega\in X^{*}(T)^{+}$, le plus grand des ordres des pôles des coefficients de la matrice $\rho_{\omega}(g)$ est égal à $\left\langle \omega,-w_{0}\lambda\right\rangle$, où $w_{0}$ est l'élément long du groupe de Weyl. De plus, l'élément $g_{+}=(\pi^{-w_{0}\la},g)$ est dans $V_{G}^{0}(\mathcal{O})$ (resp. $V_{G}(\mathcal{O})$).
\end{lem}
\begin{proof}
Le plus grand des ordres des pôles est invariant à gauche et à droite par $K$, en particulier il nous suffit de regarder celui de $\pi^{\la}$, cet ordre est égal à $\left\langle \omega,-w_{0}\lambda\right\rangle$. Inversement, les entiers $\left\langle \omega,-w_{0}\lambda\right\rangle$ déterminent uniquement $\lambda$.

Enfin, comme l'élément $g_{+}$  est dans $G_{+}(F)$ et que pour tout $\omega\in X^{*}(T)^{+}$, $\rho_{\omega}(g_{+})\in \End V_{\omega}(\mathcal{O})$, la continuité nous donne le résultat voulu.
La preuve pour $\overline{K\pi^{\la}K}$ est analogue.
\end{proof}
D'après \cite[Prop. 2]{H}, on peut donner une structure d'ind-schéma sur $k$ au quotient $G(F)/K$ qui s'écrit comme limite inductive de variétés projectives, munies d'immersions fermées les unes dans les autres.

De plus, d'après \cite[Prop. 2]{H}, ce quotient représente le foncteur $\textbf{Q}$ qui à toute $k$-algèbre $R$, associe le groupoïde $\textbf{Q}(R)$ des $G$-torseurs $E$ sur $\Spec(\mathcal{O})\hat{\times} R$ muni d'une trivialisation sur $\Spec (F)\hat{\times} R$. Ici, $\Spec(\mathcal{O})\hat{\times} R$  désigne la complétion $\pi$-adique de $\Spec (\mathcal{O})\times R$ et $\Spec (F)\hat{\times} R$ est l'ouvert complémentaire de $\{\pi\}\times R$ dans $\Spec(\mathcal{O})\hat{\times} R$.

Quant aux strates, elles admettent également une interprétation modulaire.
Pour $1\leq i\leq r$, soit $(\rho_{i},V_{i})$ la représentation irréductible de plus haut poids $\omega_{i}$.
Soit $E$ un $G$ torseur, on peut donc pousser le $G$-torseur par la représentation $\rho_{i}$. On obtient alor un fibré vectoriel noté $\rho_{i}(E)$.
 
Pour un cocaractère dominant $\lambda$, on considère le sous-foncteur $\textbf{Q}^{\lambda}$, qui à toute $k$-algèbre $R$, associe le groupoïde des $G$-torseurs $E$ sur $\Spec (\mathcal{O})\hat{\times} R$ munis d'une trivialisation générique telle que pour tout $i=1,\dots,r,$ la trivialisation fournit une application injective de fibrés vectoriels:
\begin{center}
$\rho_{i}(E)\rightarrow V_{i}(\left\langle \omega_{i},-w_{0}\lambda\right\rangle)$
\end{center}
dont les fibres résiduelles sont non nulles et $V_{i}(\left\langle \omega_{i},-w_{0}\lambda\right\rangle):=V_{i}\otimes_{R[[\pi]]} \pi^{\left\langle \omega_{i},w_{0}\lambda\right\rangle}R[[\pi]]$. D'après le lemme précédent, les $k$-points de ce foncteur sont  $\Gr_{\lambda}=K\pi^{\lambda}K/K$.
Si l'on  n'impose pas que les fibres résiduelles soient non nulles, on trouve l'espace
$\overline{\Gr}_{\lambda}=\coprod\limits_{\mu\leq\lambda}\Gr_{\mu}$.

La strate $\Gr_{\lambda}$ est lisse, mais ce n'est plus du tout le cas de son adhérence, qui est toutefois projective.

\subsection{Une interprétation modulaire}
On s'intéresse à une variante de la fibre de Springer affine:
\begin{center}
$\{g\in G(F)/K\vert ~g^{-1}\gamma g \in K\pi^{\lambda}K\}$ 
\end{center}
pour $\gamma\in G(F)$. 
Nous voulons  voir cette nouvelle fibre de Springer affine, comme une solution à un problème d'espace de modules.

On rappelle que le semi-groupe de Vinberg $V_{G}$ s'obtenait comme la normalisation de l'adhérence de 
\begin{center}
$G_{+}=(T\times G)/Z_{G}$ dans $\prod\limits_{i=1}^{r} \End (V_{\omega_{i}})\times\prod\limits_{i=1}^{r}\mathbb{A}^{1}_{\alpha_{i}}$.
\end{center}
La donnée de $-w_{0}\lambda\in X_{*}(T)^{+}=T(F)/T(\mathcal{O})$ fournit un $T$-torseur $T_{-w_{0}\lambda}$ muni d'une trivialisation générique. Cela revient également à pousser le $\mathbb{G}_{m}$-torseur $\pi^{-1}\co$  par $-w_{0}\lambda:\mathbb{G}_{m}\rightarrow T$.
De plus, pour chaque $i=1,\dots,r$, la racine $\alpha_{i}:T\rightarrow\mathbb{G}_{m}$ permet de tordre la droite $\mathbb{A}^{1}_{\alpha_{i}}$ par le $T$-torseur $T_{-w_{0}\lambda}$. On pose $b_{i}:=1_{(\left\langle \alpha_{i}, -w_{0}\lambda\right\rangle)}$ la section unité. 
Posons $X=\Spec(\mathcal{O})$. On considère l'espace caractéristique $\mathfrak{C}_{+}:=V_{T}/W$ où $V_{T}$ est l'adhérence dans le semi-groupe de Vinberg de $T_{+}$.
Nous avons également introduit dans la section \ref{rapsemi} le morphisme d'abélianisation,
\begin{center}
$\alpha:V_{G}\rightarrow A_{G}=\mathbb{A}^{r}$,
\end{center}
où $r=\rg G$.
Le morphisme d'abélianisation consiste en le quotient, au sens des invariants, de $V_{G}$ par l'action de $G\times G$ par translation à gauche et à droite (cf. sect. \ref{rapsemi}).
Le $T$-torseur $T_{-w_{0}\lambda}$  avec les sections $(b_{1},\dots,b_{r})$ obtenues en poussant le $T$-torseur par les racines simples revient alors à la donnée d'une flèche:
\begin{center}
$h_{-w_{0}\la}:X\rightarrow[A_{G}/Z_{+}]$,
\end{center}
on rappelle que $Z_{+}$ s'identifie au tore et on fait agir $Z_{+}$ par les racines simples sur $A_{G}$.
Le morphisme de Steinberg étant $Z_{+}$-équivariant, on a une flèche: 
\begin{center}
$\chi_{+}:[V_{G}/Z_{+}]\rightarrow [\mathfrak{C}_{+}/Z_{+}]$.
\end{center}
La section de Steinberg va nous permettre d'obtenir une section à ce morphisme. Comme nous l'avons vu dans la proposition \ref{action}, quitte à extraire une racine $c$-ième où $c=\left|Z_{G}\right|$, on peut faire en sorte que la section $\eps_{+}$ soit équivariante pour l'action tordue de $Z_{+}$. 

Soit le champ quotient $[\kc/Z_{+}]$ pour l'action naturelle de $Z_{+}$ sur $\kc$.
Etant donné un $T$-torseur $E$ sur un $k$-schéma $X$, il s'écrit comme une somme de fibrés en droites $E=\bigoplus\limits_{i=1}^{r}L_{i}$. 
Nous appelerons une racine $c$-ième de $E$, tout $T$-torseur $E'$ sur $X$ tel que $E'=\bigoplus\limits_{i=1}^{r}L_{i}'$ avec $(L_{i}')^{\otimes c}=L_{i}$.
On suppose donc $\la=c\la'$.
Le morphisme de Steinberg $\chi_{+}$ étant $Z_{+}$-équivariant pour les actions canoniques, il induit un morphisme sur les champs quotients:
\begin{center}
$[\chi_{+}]:[V_{G}/(G\times Z_{+})]\rightarrow[\kc/Z_{+}]$.
\end{center}
La proposition est alors la suivante :
\begin{prop}\label{racine}
Soit $S$ un $k$-schéma muni d'un $T$-torseur, $h_{-w_{0}\la}:S\rightarrow BT$ le morphisme vers le classifiant du tore $T$ associé. Soit $a:S\rightarrow[\kc/Z_{+}]$, alors la section de Steinberg et le choix d'une racine $c$-ième $h_{-w_{0}\la'}$ de $h_{-w_{0}\la}$ définit une section:
\begin{center}
$\eps_{+}(a)^{\la'}:\kc\rightarrow[V_{G}/(G\times Z_{+})]$.
\end{center}
au morphisme $[\chi_{+}]$.
\end{prop}
\begin{proof}
On note $[\chi_{+}]^{[c]}:[V_{G}/(G\times Z_{+})]^{[c]}\rightarrow[\kc/Z_{+}]^{[c]}$ le morphisme obtenu en élevant à la puissance $c$ les actions canoniques de $Z_{+}$ sur $V_{G}$ et sur $\kc$.
Nous avons vu en vertu de la proposition \ref{action} que nous avons une section,
\begin{center}
$[\mathfrak{C}_{+}/Z_{+}]^{[c]}\rightarrow[V_{G}^{reg}/(Z_{+}\times Z_{+}^{\tau})]^{[c]}$
\end{center}
où le premier $Z_{+}$ agit par homothétie et le deuxième par conjugaison par les éléments $\tau(z)$ avec $\tau:Z_{+}\rightarrow T_{\Delta}$ construit dans la proposition \ref{action}. En composant alors par ce morphisme $\tau$, nous obtenons une flèche:
\begin{center}
$[\mathfrak{C}_{+}/Z_{+}]^{[c]}\rightarrow[V_{G}^{reg}/(Z_{+}\times G)]^{[c]}$,
\end{center}
qui est donc une section au morphisme de Steinberg:
\begin{center}
$[\chi_{+}]^{[c]}:[V_{G}^{reg}/(Z_{+}\times G)]^{[c]}\rightarrow[\mathfrak{C}_{+}/Z_{+}]^{[c]}$, 
\end{center}
ce qu'on voulait.
\end{proof}
Nous pouvons désormais donner une interprétation modulaire de la fibre de Springer affine. Pour chaque flèche $h_{a}$ qui rend le diagramme :
$$\xymatrix{X\ar[dr]_{h_{-w_{0}\lambda}}\ar[r]^-{h_{a}}&[\mathfrak{C}_{+}/Z_{+}]\ar[d]^{p_{1}}\\&[\mathbb{A}^{r}/Z_{+}]}$$
commutatif, on a un point :
\begin{center}
$[\epsilon_{+}](a)\in[V_{G}^{reg}/(G\times Z_{+})]$.
\end{center}
Nous faisons alors la définition suivante:
\begin{defi}\label{sophist}
On définit la fibre de Springer affine $\mathcal{M}_{\lambda}(a)$ (resp. $\mathcal{M}_{\lambda}^{reg}(a)$) comme le foncteur dont le groupoïde des $R$-points  pour une $k$-algèbre $R$ est:
$$\xymatrix{X\hat{\times}R\ar[ddrr]_{h_{-w_{0}\lambda}}\ar[rr]^{h_{E,\phi_{+}}}\ar[drr]^{h_{a}}&&[V_{G}^{0}/G\times Z_{+}]\ar[d]^{\chi_{+}}\\&&[\mathfrak{C}_{+}/Z_{+}]\ar[d]^{p_{1}}\\&&[\mathbb{A}^{r}/Z_{+}]}$$ 
muni d'un isomorphisme entre la restriction de $h_{E,\phi_{+}}$ à $X^{\bullet}\hat{\times} R$ et la section de Steinberg $[\epsilon_{+}]^{\la'}(a)\in[V_{G}^{reg}/G]$, (resp. les morphismes $h_{E,\phi_{+}}$ qui se factorisent par $[V_{G}^{reg}/G\times Z_{+}])$.
\end{defi}
\begin{lem}
Si on regarde les $k$-points de $\mathcal{M}_{\lambda}(a)$, en négligeant les nilpotents, il s'agit de l'ensemble suivant:
\begin{center}
$\{g\in G(F)/K\vert ~g^{-1}\gamma_{0} g \in V_{G}^{0}(\mathcal{O})\}$.
\end{center}
où $\gamma_{0}=\epsilon_{+}(a)$. Celui-ci étant non vide par définition.
\end{lem}
$\rmq$ La condition que $h_{a}$ soit au-dessus de $h_{-w_{0}\lambda}$, nous donne que $\gamma_{0}=(\pi^{-w_{0}\la},\gamma)$, pour un certain $\gamma\in K\pi^{\lambda}K$. 
Nous verrons par la suite que si $a\in\mathfrak{C}_{+}(\mathcal{O})\cap\mathfrak{C}_{+}(F)^{rs}$, cet ensemble
est un schéma localement de type fini dont nous calculerons la dimension.
\begin{proof}
Soit $(E,\phi_{+})\in\mathcal{M}_{\lambda}(a)(k)$ avec un isomorphisme générique $\beta$ avec la section de Steinberg $(E_{0},\g_{0})$ où $E_{0}$ est le torseur trivial.
La donnée de $(E,\beta)$ nous fournit un élément $g\in G(F)/K$. Pour obtenir une section sur $\Spec(\co)$, l'isomorphisme avec la section de Steinberg nous donne $g^{-1}\g_{0}g\in V^{0}_{G}(\co)$.
\end{proof}
\subsection{Symétries d'une fibre de Springer affine}
On rappelle que $I:=\{(g,\g)\in G\times V_{G}\vert~ g\g g^{-1}=\g\}$.
Dans la section \ref{galoiscent}, nous avons défini le centralisateur régulier $J$,  dont nous avons vu qu'il était muni d'un morphisme:
\begin{center}
$\chi_{+}^{*}J\rightarrow I$, 
\end{center}
qui est un isomorphisme au-dessus de $V_{G}^{reg}$. Nous formons le carré cartésien suivant:
$$\xymatrix{\kcd\ar[d]\ar[r]&[\mathfrak{C}_{+}/Z_{+}]\ar[d]^{p_{1}}\\X\ar[r]^-{h_{-w_{0}\la}}&[\mathbb{A}^{r}/Z_{+}]}$$
On se donne alors une section $h_{a}:X\rightarrow\kcd$, nous avons l'image réciproque $J_{a}=h_{a}^{*}J$.
\begin{defi}\label{centr}
Considérons  le groupoïde de Picard $P(J_{a})$ au-dessus de $\Spec(k)$ qui associe à toute $k$-algèbre $R$, le groupoïde des $J_{a}$-torseurs sur $R[[\pi]]$, munis d'une trivialisation sur $R((\pi))$.
\end{defi}
On définit une action du champ $P(J_{a})$ sur $\mathcal{M}_{\lambda}(a)$. En effet si $(E,\phi_{+})\in\mathcal{M}_{\lambda}(a)(R)$, on a un morphisme de faisceaux:
\begin{center}
$J_{a}\rightarrow\underline{\Aut}(E,\phi_{+})$ 
\end{center}
qui se déduit de la flèche $\chi_{+}^{*}J\rightarrow I$.
Celui-ci permet de tordre $(E,\phi_{+})$ par un $J_{a}$-torseur sur $X\hat{\times} R$ trivialisé sur $X^{\bullet}\hat{\times} R$.
On a alors la proposition suivante: 
\begin{prop}\label{picardloc}
$[\chi_{+}]:[V_{G}^{reg}/(G\times Z_{+})]\rightarrow[\mathfrak{C}_{+}/Z_{+}]$ est une gerbe liée par le centralisateur $J$ et neutre.

En particulier, la fibre de Springer $\mathcal{M}_{\lambda}^{reg}(a)$ est un espace principal homogène sous $\mathcal{P}(J_{a})$. Dans la suite, c'est ce qu'on appellera l'orbite régulière.
\end{prop}
\begin{proof}
Nous avons vu que le morphisme était lisse. De par la caractérisation du centralisateur régulier, le faisceau des automorphismes d'un élément $(E,\phi_{+})\in [V_{G}^{reg}/(G\times Z_{+})](S)$ au-dessus d'un point $a:S\rightarrow[\mathfrak{C}_{+}/Z_{+}]$ est canoniquement isomorphe à $a^{*}J$. La neutralité vient de l'existence de $\epsilon_{+}$.
\end{proof}
De même, on considère le schéma $V_{T}^{\la}$ qui s'obtient de la même manière en considérant la flèche $\alpha_{\vert V_{T}}:V_{T}\rightarrow\ab^{r}$.
On considère alors la flèche finie plate surjective, génériquement étale:
\begin{center}
$\theta:V_{T}^{\la}\rightarrow\kcd$.
\end{center}
Nous avons introduit le diviseur discriminant $\mathfrak{D}_{+}=2\rho.\mathfrak{D}\subset\kcd$ dans la proposition \ref{prongdisc} dont le lieu de non-annulation s'identifie au lieu régulier semisimple.
On le tire alors sur la base $\kcd$ en un diviseur noté $\mathfrak{D}_{\la}$.
Pour $a\in\kcd(\co)\cap\kc^{rs}(F)$, on pose $d(a):=\val(a^{*}\mathfrak{D}_{\la})$.
Nous allons donner une formule pour $d(a)$. Soit $t_{+}=(\pi^{-w_{0}\la},t)\in V_{T}^{\la}(\overline{F})$ tel que $\theta(t_{+})=a$ et où $\overline{F}$ est la clôture algébrique de $F$ et $\overline{\co}$ son anneau d'entiers.
Par le critère valuatif, on a que $t_{+}\in V_{T}^{\la}(\overline{\co})$ et on déduit:
\begin{equation}
d(a)=\left\langle 2\rho,\la\right\rangle+\val(\det(\Id-\ad(t):\mathfrak{g}(F)/\mathfrak{g}_{t}(F)\rightarrow\mathfrak{g}(F)/\mathfrak{g}_{t}(F))).
\label{calculvin}
\end{equation}

\subsection{Le théorème principal de dimension}
On suppose $G$ semisimple simplement connexe avec $F=k((\pi))$ et $k$ algébriquement clos.
On note $G^{rs}\subset G$ l'ouvert constitué des éléments réguliers semisimples.
On considère la fibre de Springer:
\begin{center}
$X_{\gamma}^{\lambda}=\{g\in G(F)/K\vert ~g^{-1}\gamma g\in K\pi^{\lambda}K\}$.
\end{center}
On pose $d(\gamma):=\val(\det_{F}(\Id-\ad(\gamma):\mathfrak{g}(F)/\mathfrak{g}_{\gamma}(F)\rightarrow\mathfrak{g}(F)/\mathfrak{g}_{\gamma}(F)))$ et  $\defa(\gamma)=\rg G- \rg_{F}G_{\gamma}(F)$ où $\rg_{F}$ désigne le rang du plus grand sous-tore déployé du centralisateur $G_{\g}(F)$  de $\g$ et $\mathfrak{g}_{\gamma}(F)$ son algèbre de Lie.
\begin{thm}\label{dim}
Soit $\gamma\in G(F)^{rs}$. On suppose $X_{\gamma}^{\lambda}$ non vide.
\begin{enumerate} 
\item $X_{\gamma}^{\lambda}$ est  un schéma localement de type fini.
\item $\dim X_{\gamma}^{\lambda}=\left\langle \rho,\lambda\right\rangle + \frac{1}{2}[d(\gamma)-\defa(\gamma)]$.
\end{enumerate}
\end{thm}
On démontre ce théorème dans la section \ref{finrefdim}.

Regardons comment une fibre de Springer est reliée au semi-groupe de Vinberg. Considérons l'élément $\gamma_{+}=(\pi^{-w_{0}\la}, \gamma)\in G_{+}(F)$ et la fibre de Springer:
\begin{center}
$X_{\gamma_{+}}=\{g\in G(F)/G(\mathcal{O})\vert ~g^{-1}\gamma_{+}g\in V_{G}^{0}(\mathcal{O})\}$.
\end{center}
Nous avons le lemme suivant:
\begin{lem}\label{renorm}
La variété $X_{\gamma}^{\lambda}$ s'identifie canoniquement à  $X_{\gamma_{+}}$.
\end{lem}
\begin{proof}
La variété $X_{\gamma}^{\lambda}$ peut s'interpréter de la manière suivante:
\begin{center}
$X_{\gamma}^{\lambda}=\{g\in G(F)/K\vert~\forall i, \rho_{\omega_{i}}(g)^{-1}\pi^{\left\langle \omega_{i},-w_{0}\lambda\right\rangle}\rho_{\omega_{i}}(\gamma)\rho_{\omega_{i}}(g)\in \M_{n_{i}}(\mathcal{O})\cap \GL_{n_{i}}(F)\}$.
\end{center}
Ainsi, nous reconnaissons en l'élément $\pi^{\left\langle \omega_{i},-w_{0}\lambda\right\rangle}\rho_{\omega_{i}}$ l'image de $\gamma_{+}$ par $\rho_{(\omega_{i},\omega_{i})}$ et il résulte du lemme \ref{cont} que $\gamma_{+}$ est dans $V_{G}^{0}(\mathcal{O})$.
\end{proof}
On a le corollaire du théorème \ref{dim}:
\begin{cor}
Soit $a\in\kcd(\co)\cap\kc^{rs}(F)$, on considère la fibre de Springer $\cm_{\la}(a)=X_{\eps_{+}(a)}$ introduite dans la définition \ref{sophist}, alors nous avons:
\begin{center}
$\dim\cm_{\la}(a)=\frac{d(a)-c(a)}{2}$.
\end{center}
où $d(a)$ est la valuation du discriminant sur la base $\kcd$ et $c(a)=\rg(T)-\rg_{F}(J_{a}(F))$.
\end{cor}
\begin{proof}
En effet, l'élément $\eps_{+}(a)$ s'écrit $(\pi^{-w_{0}\la}, \gamma)$ pour un certain $\g\in G(F)^{rs}$. On applique alors la formule de dimension du théorème \ref{dim} ainsi que l'équation \eqref{calculvin} qui exprime le discriminant de la base $\kcd$ en fonction du discriminant sur $G$ et de $\la$.
\end{proof}
On veut savoir quand est-ce que la fibre $X_{\g}^{\la}$ est non vide. On commence par quelques rappels sur les points de Newton. Soit $T$ un tore sur $F$, soit $a\in T(\bar{F})$, on définit l'élément $\nu_{a}\in X_{*}(T)_{\mathbb{Q}}$ en imposant que:
\begin{center}
$\forall~ \lambda\in X^{*}(T)$, $\left\langle \lambda,\nu_{a} \right\rangle=\val(\lambda(a))$.
\end{center}
On obtient donc une flèche surjective $\nu_{T}:T(\bar{F})\rightarrow X_{*}(T)_{\mathbb{Q}}$ et en prenant les invariants sous Galois, nous obtenons une flèche:
\begin{center}
$\nu:T(F)\rightarrow X_{*}(A_{T})_{\mathbb{Q}}$,
\end{center}
où $A_{T}$ est le plus grand sous-tore déployé sur $F$ de $T$.
Soit $\gamma\in G(F)^{rs}$ régulier semi-simple. En ce cas, son centralisateur est un tore maximal $T$ défini sur $F$ et on pose: 
\begin{center}
$\nu_{\gamma}:=\nu_{T}(\gamma)$.
\end{center}
Notons $[\nu_{\gamma}]$, sa classe de conjugaison sous $W$, représentée par un élément dominant de $X_{*}(A_{T})_{\mathbb{Q}}$, on l'appelle le point de Newton de $\gamma$.
\begin{thm}\label{vacuite}
Soit $\gamma\in G(F)^{rs}$. Alors $X_{\gamma}^{\lambda}$ est non vide si et seulement si $[\nu_{\gamma}]\leq\lambda$.
\end{thm}
$\rmq$ Par Kottwitz-Viehmann \cite[Cor. 3.6]{KV}, on sait que si la fibre est non vide alors $[\nu_{\gamma}]\leq\lambda$. Il nous faut donc voir la réciproque.

\begin{proof}
Posons $\g_{+}:=(\pi^{-w_{0}\la},\g)$. Comme $[\nu_{\gamma}]\leq\lambda$, il résulte de \cite[Thm. 1.5.2]{Kot4}, que $a=\chi_{+}(\g_{+})\in\kcd(\co)$.
On considère alors $\g_{0}:=\eps_{+}(a)$. Il résulte de la proposition \ref{racine} que la fibre de Springer de $\g_{0}$ est non vide.
Montrons que $\g_{+}$ et $\g_{0}$ sont conjugués ce qui montrera le théorème.
On considère le schéma sur $k((\pi))$ suivant:
\begin{center}
$\mathcal{T}:=\{h\in G(F)\vert~ h^{-1}\g_{+}h=\g_{0}\}$
\end{center}
C'est un torseur sous le centralisateur $G_{\g_{+}}(F)$ qui est un tore et comme $\g_{+}$ est régulier semisimple et que nous avons $\chi_{+}(\g)=\chi_{+}(\g_{0})$, il admet localement des sections. D'après le théorème de Lang, tout torseur sous un schéma en tores sur $k((\pi))$ est trivial, on obtient donc un élément qui conjugue.
\end{proof}

\section{Preuve du théorème principal}\label{casgen}

\subsection{Décomposition de Jordan topologique}
Nous dirons qu'un élément $\gamma\in G(F)$ est compact s'il engendre un sous-groupe relativement compact de $G(F)$. Nous dirons qu'il est topologiquement unipotent si $\gamma^{p^{r}}\rightarrow 1$ quand $r\rightarrow +\infty$ dans $\End_{F}(V)$ où la topologie est celle induite par le corps $F$ et $V$ une représentation fidèle. Tout élément topologiquement unipotent est compact.

De plus, un élément $\gamma'\in G(F)$ est dit absolument semi-simple s'il est d'ordre fini premier à $p=\car k$. 
Enfin, on se donne $Z(F)$ est un sous-groupe fermé normal de $G(F)$.
Nous dirons qu'un élément est topologiquement unipotent modulo $Z$ (resp. absolument semi-simple modulo $N$), si son image dans $G(F)/Z(F)$ est topologiquement unipotente (resp. absolument semi-simple). On définit de même la notion de compacité modulo $Z$.
On a alors le théorème suivant dû à Spice \cite[Prop. 2.41]{Spi} :
\begin{thm}
Si $\gamma\in G(F)$ est compact modulo $Z$, alors il admet une décomposition de Jordan topologique modulo $Z$, i.e. on a l'écriture:
\begin{center}
$\gamma=\gamma_{s}\gamma_{u}$ 
\end{center}
avec $\gamma_{s}$ absolument semi-simple modulo $Z$ et $\gamma_{u}$ topologiquement unipotent modulo $Z$ qui commutent entre eux et avec $\gamma$.
Si $Z$ est trivial, il y a de plus unicité.
\end{thm}
Nous appliquons donc ce théorème pour obtenir la proposition suivante:
\begin{prop}
Soit $\gamma\in G(F)$ régulier semi-simple, alors il existe un groupe de Lévi $M$ tel que $\g\in M(F)$ et admette une décomposition de Jordan modulo un certain groupe $Z_{M}(F)$.
\end{prop}
\begin{proof}
L'élément $\gamma$ étant régulier semi-simple, s'il est anisotrope, il est compact et donc on peut prendre pour $Z(F)$ le groupe trivial, sinon l'élément $\gamma$ est  elliptique dans un certain Lévi $M(F)$ et donc compact modulo le centre de ce Lévi $Z_{M}(F)$.
\end{proof}
On rappelle le théorème de Kottwitz-Viehmann de descente démontré sur $\mathbb{C}$, mais dont la preuve vaut en caractéristique $p$ également \cite[Th. 3.5]{KV}:

\begin{thm}\label{kv}
Soit $P$ un parabolique, on écrit $P=MN$ où $M$ est le Lévi.
Soit $\gamma\in M(F)$ et $X_{M}^{\lambda}$, la fibre de Springer pour ce Lévi, alors si $\Ad(\gamma)$ agit sur $\Lie(N)$ avec des pentes strictement positives (cf. \cite[sect. 2.1]{KV}), l'injection canonique $X_{M,\gamma}^{\lambda}\rightarrow X_{\gamma}^{\lambda}$ est une bijection.
\end{thm}

$\rmq$ Le corollaire 3.6 de \cite{KV} nous assure que le point de Newton définit un parabolique $P$, sur lequel $\ad(\gamma)$ agit sur $\Lie(N)$ via son point de Newton, avec des pentes positives.

D'après ci-dessus, nous pouvons donc considérer la décomposition de Jordan topologique de notre élément $\gamma\in G(F)^{rs}$. Alors, $\gamma=\gamma_{s}\gamma_{u}$ avec $\gamma_{u}$ topologiquement unipotent et $\gamma_{s}$ fortement semi-simple modulo le centre d'un certain Lévi $M(F)$. 

En utilisant Kottwitz-Viehmann, on sait que $X_{\gamma,M}^{\lambda}=X_{\gamma}^{\lambda}$. On se ramène donc au cas du Lévi. 
Posons $H=C_{G}(\gamma_{s})$. On peut  choisir un sous-groupe parabolique $P'$ de $M$ dont le facteur de Lévi est $H=C_{G}(\gamma_{s})$ et on réapplique le théorème de Kottwitz-Viehmann.
On obtient alors un isomorphisme:
\begin{center}
$X_{\gamma}^{\lambda}= X_{\gamma_{u}}^{\lambda,H}$.
\end{center}
En particulier, on peut supposer que $\gamma_{s}=1$ et donc $\gamma=\gamma_{u}$, ce que l'on fera dans la suite de ce travail.

On pose $\g_{+}=(\pi^{-w_{0}\la},\g)$, comme la fibre de Springer $X_{\g_{+}}=X_{\g}^{\la}$ est non vide par hypothèse, quitte à conjuguer, on peut supposer que $\g\in V_{G}^{0}(\co)$.
On rappelle que:
\begin{center}
$X_{\g_{+}}:=\{g\in G(F)/G(\co)\vert~g^{-1}\g_{+}g\in V_{G}^{0}(\co)\}$.
\end{center}

Un élément $M\in\End(V)$ est dit \textit{quasi-unipotent} si ses valeurs propres sont 0 ou 1 (le lieu d'annulation du polynôme $X^{n}(X-\Id)^{n}$).
On dit d'un élément de $g\in V_{G}(\co)$ qu'il est toplogiquement quasi-unipotent si $g^{p^{r}}$ tend vers un élément quasi-unipotent quand $r\rightarrow+\infty$ dans $\End_{F}(V)$.
En particulier, comme $\g_{u}$ est topologiquement unipotent modulo $Z(F)$, l'élément $\g_{+}$ est topologiquement quasi-unipotent. 
L'élément $\g_{+}$ est dit toplogiquement nilpotent s'il est toplogiquement quasi-unipotent de limite nulle.
Pour obtenir le théorème \ref{dim}, suivant Kazhdan-Lusztig \cite{KL}, nous avons besoin de démontrer l'équidimensionnalité d'une certaine variété de drapeaux associée. Comme l'élément $\g_{+}$ est topologiquement quasi-unipotent, il admet une partie nilpotente et une partie unipotente.
Commençons par rappeler comment on traite la partie unipotente.

\subsection{Le cas topologiquement unipotent}\label{droite}
L'élément $\g_{+}$ est topologiquement unipotent si et seulement si $\la=0$. Dans ce cas, on a $\g=\g_{+}\in G(\co)$.
On a une flèche $\ev:K\rightarrow G$ et on pose $I:=\ev^{-1}(B)$ et $\cB=G(F)/I$, la variété de drapeaux affine associée qui classifie les Iwahori $\hat{B}$.

On considère alors  la variété $\cB_{\g}:=\{g\in \cB\vert~ g^{-1}\g g\in I\}$.
Soit $\hat{\Delta}$ l'ensemble des racines simples du groupe de Weyl affine.
Pour chaque racine $\alpha\in\hat{\Delta}$, on a un parabolique $\hat{P}_{\alpha}:=I\cup Is_{\alpha}I$. Soit le ind-schéma $\mathcal{P}_{\alpha}:=G(F)/\hat{P}_{\alpha}$, on a un morphisme naturel $\pi_{\alpha}:\mathcal{B}\rightarrow\mathcal{P}_{\alpha}$ qui est une $\mathbb{P}^{1}$-fibration. Ses fibres seront appelées des droites de type $\alpha$.
Pour un point $g\in\cB$, on note $\bP_{\alpha}^{1}(g)$, la droite de type $\alpha$ passant par $g$.

Kazhdan-Lusztig considèrent alors pour chaque $\alpha\in\hat{\Delta}$, un fibré vectoriel $\cE_{\alpha}$ au-dessus de $\cB_{\g}$, dont la fibre au-dessus d'un Iwahori $\hat{B}$ est donnée par $R_{u}(\hat{B})/R_{u}(\hat{P}_{\alpha})$ où $R_{u}(\hat{B})$ désigne le radical pro-unipotent de $\hat{B}$.
Ils définissent alors une section $s_{\alpha}$ de ce fibré vectoriel $\cE_{\alpha}$ dont l'image en un point $\hat{B}\in\cB$ est donnée par l'image de $\g$ dans le quotient $R_{u}(\hat{B})/R_{u}(\hat{P}_{\alpha})$.
Cette section a la propriété notable suivante; soit $g\in\cB_{\g}$, alors 
\begin{center}
$s_{\alpha}(g)=0\Longleftrightarrow\bP_{\alpha}^{1}(g)\subset\cB_{\g}$.
\end{center}
L'existence de cette section nous permet d'établir la proposition suivante:
\begin{prop}\label{lemmun}
Soit $Y$ une composante irréductible de dimension maximale de $\cB_{\g}$. Soit $\cL$ une droite de type $\alpha$, incluse dans $\cB_{\g}$ et telle que $\cL\cap Y\neq\emptyset$, alors il existe une composante irréductible $Y'$ de dimension maximale de $\cB_{\g}$ telle que $\cL\subset Y'$.
\end{prop}
Ils démontrent également la propriété suivante sur $\cB_{\g}$, \cite[§4. Lem. 2]{KL}:
\begin{prop}\label{lemmedeux}
Soient $g,g'\in\cB_{\g}$, alors il existe des droites $\cL_{1},\dots, \cL_{n}$ de type $\alpha_{1},\dots,\alpha_{n}$ dans $\cB_{\g}$  telles que $g\in\cL_{1}$, $g'\in\cL_{n}$ et $\cL_{i}\cap\cL_{i+1}\neq\emptyset$.
En particulier, $\cB_{\g}$ est connexe.
\end{prop}
De ces deux propositions, il résulte formellement le résultat suivant \cite[§4. Prop. 1]{KL} :
\begin{thm}\label{kl}
La variété $\cB_{\g}$ est équidimensionnelle.
\end{thm}

Pour faire le lien avec le cas qui nous intéresse, on commence par donner une autre formulation des éléments quasi-unipotents due à Putcha.
Soit $M$ un monoïde réductif, normal et intègre (cf. \eqref{rapsemi}), de groupe des inversibles $G_{+}$.
Soit un idempotent $e=e^{2}\in M$, on considère alors $H(e)$ le groupe des inversibles du monoïde algébrique $eMe$.
Ce groupe s'identifie à $eC_{G_{+}}(e)$. Il peut éventuellement être vide si $e=0$.
On définit une relation $\cH$ dite de Green par:
\begin{center}
$a\cH b$ si $aM=bM$ et $Ma=Mb$.
\end{center}
Nous avons alors d'après \cite[sect. 2]{Pu2},  qu'un élément $u\in M$ est quasi-unipotent s'il existe un entier $i$ et un idempotent $e=e^{2}\in M$ tel que
$u^{i}\cH e$ et $ue$ est unipotent dans le groupe $H(e)$. Si $M$ admet un zéro, alors tout élément nilpotent est clairement quasi-unipotent.

En particulier, si $\g_{+}^{p^{n}}\rightarrow e$, nous avons que  $\g_{+}e\in H(e)(\co)$ et qui est topologiquement unipotent, ce qui nous permettra d'isoler dans la fibre de Springer une partie `unipotente', pour se concentrer ensuite sur la partie nilpotente.
Nous allons maintenant montrer dans la section suivante comment se ramener à ce cas:

\subsection{Réduction au cas topologiquement nilpotent}
On a une flèche $\ev:V_{G}(\co)\rightarrow V_{G}$ et on pose $I^{\bullet}=\ev^{-1}(V_{B})$.
En particulier, quitte à changer de Borel, on peut supposer que $\g_{+}\in I^{\bullet}$.
Comme $\g_{+}$ est topologiquement quasi-unipotent, nous avons $\g_{+}^{p^{n}}\rightarrow e$ avec $e^{2}=e\in E(V_{T})$ où $E(V_{T})$ est l'ensemble des idempotents de $V_{T}$.

De plus, on a $\g_{+}e=e\g_{+}$, on note alors $G(e):=C_{G}(e)$ le centralisateur de $e$ dans $G$ et $G_{+}(e)=C_{G_{+}}(e)$.
Soit $\overline{G}^{0}_{+}(e)$ l'adhérence de $G_{+}(e)$ dans $V_{G}^{0}$.

Nous avons une description des éléments idempotents de $E(V_{T})$ d'après \cite[p. 434]{Pu2}  et Vinberg \cite[Thm. 7]{Vi} :
\begin{center}
$E(V_{T})=\coprod\limits_{w\in W}\coprod\limits_{(I,J)\in(\Delta\times\Delta)^{*}}we_{I,J}w^{-1}$,
\end{center}
où $(\Delta\times\Delta)^{*}$ est le sous-ensemble des racines simples $\Delta\times\Delta$ qui correspond aux paires $(I,J)$ essentielles (cf. \cite[Déf. 4]{Vi} ).
Si $I=J=\emptyset$, nous avons $e_{I,J}=0$.

Il résulte alors de la description des stabilisateurs des $e_{I,J}$ \cite[Thm. 7]{Vi}  et \cite[Thm. 21]{Ri}  que le centralisateur d'un idempotent $we_{I,J}w^{-1}$ est donné par: 
\begin{center}
$C_{G_{+}}(we_{I,J}w^{-1})=wL_{M}w^{-1}$, 
\end{center}
où $L_{M}$ est le Lévi associé à $M$ et 
\begin{center}
$M=(I\cap J^{0})\cup \mathstrut^{c}J$,
\end{center}
où $\mathstrut^{c}J$ est le complémentaire de $J$ dans $\Delta$  et $J^{0}\subset J$ l'intérieur de $J$, i.e. le sous-ensemble de $J$ qui consiste en les éléments dont les arêtes correspondantes dans le diagramme de Dynkin ne sont pas  adjacentes à une arête correspondant à un élément de $\mathstrut^{c}J$.
On remarque que si $e_{I,J}\in\overline{Z}_{+}$, alors on a $M=\Delta$ et $L_{M}=G$.
Dans la suite, on peut supposer sans restreindre la généralité que $e=e_{I,J}$ pour une certaine paire $(I,J)$.

On considère alors la fibre de Springer analogue dans le Lévi $G(e):=C_{G}(e)$ et en appliquant le théorème \ref{kv}, on se ramène à étudier la fibre de Springer pour $G(e)$:
\begin{center}
$X_{\g_{+}}=X_{\g_{+}}^{G(e)}:=\{g\in G(e)(F)/G(e)(\co)\vert~g^{-1}\g_{+}g\in \overline{G}^{0}_{+}(e)(\co)\}$.
\end{center}
et $e\in \overline{Z}_{G_{+}(e)}$.
Nous avons alors dans $G(e)$ un sous-groupe distingué:
\begin{center}
$G_{e}:=\{g\in G(e)\vert~ ge=eg=e\}^{0}$
\end{center}
et en considérant le groupe $G^{e}:=C_{G(e)}(G_{e})^{0}$, nous obtenons une décomposition de $G(e)$ d'après \cite[Thm. 27.5]{Hum} :
\begin{center}
$G(e)=G_{e}G^{e}=G^{e}G_{e}$.
\end{center}
De plus, nous avons $G_{e}\cap G^{e}=Z(G_{e})=Z(G^{e})$, nous distinguons alors deux cas d'après \cite[Thm. 7]{Vi} :
\begin{itemize}
\item
$e=e_{I,\Delta}$ et $G_{e}=\{1\}$,
\item
$e=e_{I,J}$ avec $J\varsubsetneq\Delta$, alors $G_{e}\neq\{1\}$ est semisimple et $Z(G_{e})$ est fini.
\end{itemize}

Si $e=e_{I}$, alors on obtient que le morphisme de groupes $\theta: G(e)\rightarrow eG(e)$ est un isomorphisme et on a un isomorphisme entre les fibres de Springer affines:
\begin{center}
$X_{\g_{+}}^{G(e)}=X_{e\g_{+}}^{eG(e)}$.
\end{center}
Or, comme nous avons maintenant que $e\g_{+}\in eG(e)(\co)$, on peut appliquer le théorème \ref{kl} pour obtenir
\begin{thm}
Si $e=e_{I}$, alors la variété $\cB_{e\g_{+}}^{eG(e)}$ est équidimensionnelle.
\end{thm}
On s'intéresse maintenant au cas où $e=e_{I,J}$ avec $J\varsubsetneq\Delta$.
On a une suite exacte:
$$\xymatrix{1\ar[r]&G_{e}\ar[r]&G(e)\ar[r]^-{m}&G(e)/G_{e}\ar[r]&1}$$
La flèche $m$ induit une surjection $G^{e}\rightarrow  G(e)/G_{e}$ de noyau fini isomorphe à $Z_{G^{e}}=Z_{G_{e}}$.

En particulier, on obtient un isomorphisme canonique entre les groupes $G(e)/G_{e}$ et $eG(e)\cong eG^{e}\cong G^{e}_{ad}$. On peut donc récrire la suite exacte sous la forme:
\begin{equation}
\xymatrix{1\ar[r]&G_{e}\ar[r]&G(e)\ar[r]^{m}&eG^{e}\ar[r]&1}
\label{suitex}
\end{equation}
On note $G_{e,+}$ le sous-groupe distingué correspondant dans $G_{+}(e)$, on a vu que l'élément $e\g_{+}\in eG^{e}(\co)$, on peut donc écrire une décomposition:
\begin{center}
$\g=\g_{1}\g_{2}$
\end{center}
avec $\g_{2}\in G^{e}(\co)$ topologiquement unipotent et $\g_{1}\in \overline{G}_{e,+}(\co)$ topologiquement nilpotent dans $\overline{G}^{0}_{e,+}$ (i.e. $\g_{1}^{p^{n}}\rightarrow e$).
La suite exacte \eqref{suitex} est une extension de groupes réductifs connexes, on  considère alors la flèche de projection $G(e)\rightarrow eG^{e}$ et comme $F=k((\pi))$ à corps résiduel algébriquement clos, on obtient une flèche surjective entre les grassmaniennes affines:
\begin{center}
$\Gr_{G(e)}\rightarrow \Gr_{eG^{e}}$
\end{center}
de fibre $\Gr_{G_{e}}$. 
On en déduit alors un isomorphisme non canonique entre les grassmaniennes affines:
\begin{center}
$\eta:\Gr_{G_{e}}\times\Gr_{eG^{e}}\rightarrow\Gr_{G(e)}$.
\end{center}

\begin{prop}
La flèche $\eta$ induit un isomorphisme:
\begin{center}
$\eta:X_{\g_{1}}^{G_{e}}\times X_{e\g_{+}}^{eG^{e}}\rightarrow X_{\g_{+}}^{G(e)}$.
\end{center}

\end{prop}

$\rmqs$
\begin{itemize}
\item Un élément de $\overline{G}_{+,e}$ est une certaine matrice diagonale par blocs:
$$\begin{pmatrix}
   A & 0 \\
   0 & Id 
\end{pmatrix}.$$
En particulier, si $\g_{1}^{p^{n}}\rightarrow e$, cela veut dire que la matrice $A$ est bien topologiquement nilpotente.
\item
C'est l'énoncé auquel nous faisions allusion lorsque nous voulions décomposer la fibre de Springer en une partie unipotente, ici celle qui correspond à $eG^{e}$ et une partie, dont nous allons voir qu'elle est la partie nilpotente.
\end{itemize}
\begin{proof}
On commence par démontrer que $\eta$ induit une application au niveau des fibres de Springer.
Considérons la paire $(g_{1},g_{2})\in X_{\g_{1}}^{G_{e}}\times X_{e\g_{+}}^{eG^{e}}$ et $g\in G(F)$ l'élément canoniquement associé.
Nous avons:
\begin{center}
$g^{-1}\g_{+}g=g^{-1}\g_{1}gg^{-1}\g_{2}g$.
\end{center}
Comme $\g_{1}\in\overline{G}_{e,+}^{0}(\co)$, on a 
\begin{center}
$g^{-1}\g_{1}g=g_{1}^{-1}\g_{1}g_{1}$.
\end{center}
Or, nous avons $g_{1}\in X_{\g_{1}}^{G_{e}}$,  on en déduit donc l'intégralité de $g_{1}^{-1}\g_{1}g_{1}$.
Il nous faut voir l'intégralité de $g^{-1}\g_{2}g$.
Comme $\g_{2}\in G^{e}(\co)$, nous avons déjà que $g^{-1}\g_{2}g\in G^{e}(F)$.
Nous avons l'égalité:
\begin{center}
$g^{-1}\g_{2}ge=g_{2}^{-1}\g_{2}g_{2}\in eG^{e}(\co)$.
\end{center}
Comme la flèche $G^{e}\rightarrow eG^{e}$ est de noyau fini qui s'identifie au centre $Z_{G^{e}}$, on obtient que $g^{-1}\g_{2}g\in G^{e}(\co)$ et donc $g\in X_{\g_{+}}^{G(e)}$.
La flèche est alors bien injective. La surjectivité vient alors de l'égalité:
\begin{center}
$g^{-1}\g_{+}g=g^{-1}\g_{1}gg^{-1}\g_{2}g\in\overline{G}^{0}(e)(\co)$.
\end{center}
avec $g^{-1}\g_{1}g\in G_{e}(F)$ et $g^{-1}\g_{2}g\in G^{e}(F).$

On écrit alors $g=g_{1}g_{2}$ avec $(g_{1},g_{2})\in G_{e}(F)\times G^{e}(F)$.
Comme $G_{e}(F)\cap G^{e}(F)= Z_{G_{e}}$ qui est fini, cela force l'intégralité de $g^{-1}\g_{1}g$ et de $g^{-1}\g_{2}g\in G^{e}(F).$

A nouveau, $g^{-1}\g_{1}g=g_{1}^{-1}\g_{1}g_{1}$ et $g^{-1}\g_{2}g=g_{2}^{-1}\g_{2}g_{2}$.
Il nous suffit donc juste de considérer la paire $(g_{1},eg_{2})$.
\end{proof}

D'après le  théorème \ref{kl}, comme $e\g_{+}$ est topologiquement unipotent, on a:
\begin{prop}
La variété $\cB_{e\g_{+}}^{eG^{e}}$ est équidimensionnelle.
\end{prop} 
Il ne nous reste plus qu'à étudier la partie $X_{\g_{1}}^{G_{e}}$ avec $\g_{1}$ topologiquement nilpotent.

\subsection{L'équidimensionnalité de $\cB_{\g_{+}}$}
Dans la section précédente, nous nous sommes ramenés à l'étude d'un élément $\g_{+}\in M^{0}(\co)\cap H_{+}(F)^{rs}$ topologiquement nilpotent où $M$ est monoïde réductif, intègre,  de groupe des inversibles $H$ avec $M\subset V_{G}$ et $M^{0}:=M\cap V_{G}^{0}$. $M$ admet de plus un zéro et d'après \cite[Cor. 2.2.5]{Br}, , il est également normal.
On a la fibre de Springer associée: 
\begin{center}
$X_{\g_{+}}:=\{g\in H_{der}(F)/H_{der}(\co)\vert~g^{-1}\g_{+}g\in M^{0}(\co)\}$.
\end{center}
On se donne une paire de Borel $(B,T)$ de $M$.
Soit $\ev:M(\co)\rightarrow M$. On pose $I:=\ev^{-1}(B)$ et $I_{der}:=\ev^{-1}(B\cap H_{der})$. Si $\overline{B}$, désigne l'adhérence de $B$ dans $M$, nous définissons le semi-groupe $I^{\bullet}=\ev^{-1}(\overline{B})$ (resp. l'ouvert $I_{0}^{\bullet}:=\ev^{-1}(V_{B}^{0}))$ que nous appelons le semi-groupe d'Iwahori de groupe des inversibles $I$. 
On considère alors la variété de drapeaux :
\begin{center}
$\mathcal{B}_{\gamma_{+}}=\{g\in H_{der}(F)/I_{der}\vert~g^{-1}\gamma_{+}g\in I_{0}^{\bullet}\}$,
\end{center}
dont nous voulons démontrer l'équidimensionnalité.

Soit $\hat{\Delta}$ l'ensemble des racines simples du groupe de Weyl affine.
Pour chaque racine $\alpha\in\hat{\Delta}$, on a un semi-groupe parabolique $\hat{P}_{\alpha}^{\bullet}:=I^{\bullet}\cup I^{\bullet}s_{\alpha}I^{\bullet}$,
de groupes des inversibles $\hat{P}_{\alpha}:=I\cup Is_{\alpha}I$.

D'après \cite[Thm. 2.6]{Pu2} , on a un morphisme de monoïdes:
\begin{center}
$\phi: \overline{B}\rightarrow \overline{T}$,
\end{center}
qui prolonge la flèche de projection $B\rightarrow T$.
On considère alors le monoïde algébrique:
\begin{equation}
N:=\phi^{-1}(0).
\label{equnip}
\end{equation}
stable par conjugaison par $B$.
C'est un \textit{idéal} de $\overline{B}$, i.e. on a :
\begin{equation}
\forall~ b \in \overline{B}, bN\subset N ~\text{et}~ Nb\subset N.
\label{idealj}
\end{equation}
On pose alors $R:=\ev^{-1}(N)$ qui admet également une structure de monoïde.
Pour une racine affine simple $\alpha$, on considère alors $R_{\alpha}:=R\cap s_{\alpha}R s_{\alpha}\subset\hat{P}_{\alpha}^{\bullet}$.
On appelle $R_{\alpha}$ le pro-radical nilpotent de $\hat{P}_{\alpha}^{\bullet}$.
Plus généralement, pour un élément $w\in W_{aff}$, on définit de même $R_{w}=R\cap w Rw^{-1}$.
\begin{lem}\label{necracine}
Soit $g\in\cB_{\g_{+}}$, on suppose que $\bP^{1}_{\alpha}(g)\subset\cB_{\g_{+}}$, alors $\g_{+}\in \mathstrut^{g}R_{\alpha}$.
\end{lem}
\begin{proof}
Soit $\g':=g^{-1}\g_{+}g$, comme $g\in\cB_{\g_{+}}$, on a $\g'\in I^{\bullet}$. Les semi-groupes d'Iwahori de $\hat{P}_{\alpha}$ sont de la forme $bs_{\alpha}I^{\bullet}s_{\alpha}b^{-1}$ avec $b\in I$. Comme $\bP^{1}_{\alpha}(g)\subset\cB_{\g_{+}}$, on obtient que $\g'\in bs_{\alpha}I^{\bullet}s_{\alpha}b^{-1}$ pour tout $b\in I$.
Ainsi, on obtient en particulier $\g'\in s_{\alpha}I^{\bullet}s_{\alpha}$ et donc $\g'\in R_{\alpha}$, ce qu'on voulait.
\end{proof}

Pour pouvoir établir l'équidimensionnalité de $\cB_{\g_{+}}$, il nous faut définir une section qui nous permette d'obtenir un analogue de la proposition \ref{lemmun}.
On commence par considérer les racines simples $\alpha\in\Delta$ qui correspondent au groupe de Weyl fini $W$. 

Soit $\alpha\in\Delta$ et $P_{\alpha}$ le parabolique minimal associé à $\alpha$ et $L_{\alpha}$ son Lévi.
Soit $\overline{P}_{\alpha}$ (resp. $\overline{L}_{\alpha})$ son adhérence dans $M$.
Si l'on plonge $M$ dans $\End(V)$, pour $b\in \overline{P}_{\alpha}$, soit $\phi_{\alpha}(b)$ la matrice qui a le même facteur diagonal par bloc que la matrice $b$.
En particulier, c'est l'application identité sur $L_{\alpha}$.
Pour $x\in P_{\alpha}$, $x=lu$ avec $l\in L_{\alpha}$ et $u\in R_{u}(P_{\alpha})$.
Alors, $\phi_{\alpha}(x)=\phi_{\alpha}(l)\phi_{\alpha}(u)=l\in L_{\alpha}$.
Ainsi, $\phi_{\alpha}(\overline{P}_{\alpha})\subset\overline{L}_{\alpha}$.
On obtient donc un morphisme de monoïdes:
\begin{center}
$\phi_{\alpha}:\overline{P}_{\alpha}\rightarrow\overline{L}_{\alpha}$,
\end{center}
qui est l'identité sur $\overline{L}_{\alpha}$. On a de plus un diagramme commutatif:
\begin{equation}
\xymatrix{\overline{B}\ar[d]^{\phi}\ar[r]&\overline{P}_{\alpha}\ar[d]^{\phi_{\alpha}}\\\overline{T}\ar[r]&\overline{L}_{\alpha}}
\label{comm}
\end{equation}
où la flèche horizontale du bas est l'inclusion canonique.
On considère alors 
\begin{center}
$N_{\alpha}:=\phi_{\alpha}^{-1}(0)$.
\end{center}
A nouveau, $N_{\alpha}$ est un idéal (cf.\eqref{idealj}) de $\overline{P}_{\alpha}$. Nous avons le lemme suivant.
\begin{lem}\label{sectid}
On a  l'égalité:
\begin{center}
$N_{\alpha}=N\cap s_{\alpha}Ns_{\alpha}$,
\end{center}
et en particulier $N_{\alpha}$ est un idéal de $N$.
\end{lem}
\begin{proof}
On commence par montrer l'inclusion $N_{\alpha}\subset N\cap s_{\alpha}Ns_{\alpha}$.
Soit $x\in N_{\alpha}=\phi^{-1}_{\alpha}(0)$, il resulte de la description de l'application $\phi_{\alpha}$ dans un semi-groupe $\End(V)$ que la matrice $x$ est nécessairement triangulaire supérieure, en particulier, on a $x\in\overline{B}$ et $N_{\alpha}\subset\overline{B}$.
On a même par commutativité du diagramme \eqref{comm} que $N_{\alpha}\subset N$.
Comme de surcroît, $N_{\alpha}$ est un idéal de $P_{\alpha}$, on a:
\begin{center}
$s_{\alpha}N_{\alpha}s_{\alpha}\subset N_{\alpha}\subset N$
\end{center}
et $N_{\alpha}\subset N\cap s_{\alpha}Ns_{\alpha}$.

Montrons l'inclusion réciproque. Soit $x\in N\cap s_{\alpha}Ns_{\alpha}$, il nous faut montrer que $\phi_{\alpha}(x)=0$.
$N\cap s_{\alpha}Ns_{\alpha}$ est dans l'adhérence de $\overline{T(U\cap s_{\alpha}Us_{\alpha})}$  et nous avons que la flèche $(\phi_{\alpha})_{\vert N\cap s_{\alpha}Ns_{\alpha}}$ correspond à la partie dans $\overline{T}$ de la matrice $x$.
Or, $\phi(x)=0$, donc $\phi_{\alpha}(x)=0$, ce qu'on voulait.
Comme $N_{\alpha}$ est un idéal de $\overline{P}_{\alpha}$, c'est a fortiori un idéal de $\overline{B}$ et donc de $N$.
\end{proof}

\begin{cor}
Si $\alpha\in\Delta$, alors $R_{\alpha}\subset R$ est un idéal de $R$ et $R_{\alpha}$ est stable par conjugaison par $I$.
\end{cor}
\begin{proof}
On a tout d'abord $R_{\alpha}=\ev^{-1}(N\cap s_{\alpha}Ns_{\alpha})$ et $N_{\alpha}=\phi^{-1}_{\alpha}(0)=N\cap s_{\alpha}Ns_{\alpha}$ est un idéal de $N$ en vertu du lemme \ref{sectid}.
Enfin, $R_{\alpha}$ est stable par conjugaison par $I$ car $N_{\alpha}$ est stable par conjugaison par $B$ comme c'est un idéal de $P_{\alpha}$.
\end{proof}
Le lemme suivant va nous permettre de définir notre section:
\begin{lem}\label{radc}
Le monoïde $R_{\alpha}$ est contenu dans tous les semi-groupes d'Iwahori de $\hat{P}_{\alpha}^{\bullet}$.
\end{lem}
\begin{proof}
En effet, un semi-groupe d'Iwahori de $\hat{P}_{\alpha}^{\bullet}$ s'écrit $bs_{\alpha}I^{\bullet}s_{\alpha}b^{-1}$ avec $b\in I$.
On a que $R_{\alpha}$ est par définition stable par conjugaison par $s_{\alpha}$ et on vient de voir qu'il est stable par conjugaison par $I$, on obtient alors que :
\begin{center}
$s_{\alpha}b^{-1}xbs_{\alpha}\in R_{\alpha}\subset I^{\bullet}$,
\end{center}
d'où $x\in bs_{\alpha}I^{\bullet}s_{\alpha}b^{-1}$, ce qu'on voulait.
\end{proof}

On fait maintenant le quotient de $R$ par $R_{\alpha}$, au sens de Rees. 
Soit $S$ un semigroupe et $I$ un idéal de $S$. On considère la relation d'équivalence $\equiv$ suivante:
\begin{center}
$x\equiv y \Longleftrightarrow x=y$ ou $x,y\in I$.
\end{center}
Cette relation d'équivalence est compatible au produit du semigroupe et on forme le quotient $S/I:=S/\equiv$.

On applique  cette construction pour former le semigroupe quotient $R/R_{\alpha}$. On considère alors le fibré $\mathcal{E}_{\alpha}$ au-dessus de $\cB_{\g_{+}}^{+}$ dont la fibre en un point $g$ est donnée par le quotient $\mathstrut^{g}R/\mathstrut^{g}R_{\alpha}$.
On construit  une section $s_{\alpha}$ de $\mathcal{E}_{\alpha}$, dont l'image d'un point $g\in\cB_{\g_{+}}$ est donnée par l'image de $\g_{+}$ dans le quotient $\mathstrut^{g}R/\mathstrut^{g}R_{\alpha}$.
Comme $R_{\alpha}$ est le radical pro-nilpotent de $\hat{P}_{\alpha}^{\bullet}$, nous avons en vertu du lemme \ref{radc} pour $g\in\cB_{\g_{+}}$: 
\begin{center}
$s_{\alpha}(g)=0\Longleftrightarrow\bP^{1}_{\alpha}(g)\subset\cB_{\g_{+}}.$
\end{center}

Il ne  nous reste plus qu'à traiter le cas de la racine affine $\alpha_{0}$.
L'inconvénient de la racine affine $\alpha_{0}$ est que cette fois $R_{\alpha_{0}}$, comme le montre un calcul pour $SL_{2}$, n'est plus un idéal de $R$. Nous ne pouvons donc pas faire le quotient au sens de Rees.

Fort heureusement, le lemme suivant nous assure que la variété $\cB_{\g_{+}}$ ne contient pas de droites de type $\alpha_{0}$.

On rappelle que $s_{\alpha_{0}}=\pi^{\tilde{\alpha}\check{}}s_{\tilde{\alpha}}$ où $\tilde{\alpha}$ est l'unique racine la plus haute.
En particulier, si on écrit $\tilde\alpha=\sum m_{i}\alpha_{i}$ comme combinaison linéaire de racines simples et que l'on considère une autre racine positive $\beta=\sum p_{i}\alpha_{i}$, on a pour tout $i$, $p_{i}\leq m_{i}$.

\begin{prop}\label{inclnil}
On a l'inclusion $R_{\alpha_{0}}\subset\ev^{-1}(0)$.
\end{prop}
\begin{proof}
On a $I=T(\co)K_{0}$ où on a posé $K_{0}:=\ev^{-1}(U)$. Considérons également $K_{1}:=\ev^{-1}(1)$.
On a que $N$ est dans l'adhérence de $I$ et en particulier, $R\cap s_{\alpha_{0}}Rs_{\alpha_{0}}$ est dans l'adhérence de $T(\co)K_{0}\cap s_{\alpha_{0}}K_{0}s_{\alpha_{0}}$.
Il nous suffit donc de voir que $K_{0}\cap s_{\alpha_{0}}K_{0}s_{\alpha_{0}}\subset K_{1}$ pour conclure, comme en réduction, on tombe dans $N$.
Le groupe $K_{0}$ est engendré par le produit:
\begin{center}
$T(1+\pi\co)\times\prod\limits_{\beta>0}U_{\beta}(\co)\times U_{-\beta}(\pi\co)$.
\end{center}
Comme $s_{\alpha_{0}}$ stabilise $T(1+\pi\co)$, il nous suffit d'étudier les groupes radiciels.
Soit une racine positive $\beta$ et $k\in\mathbb{N}$, regardons comment agit $s_{\alpha_{0}}$ sur $U_{\beta}(\pi^{k}\co)$:
\begin{center}
$s_{\alpha_{0}}U_{\beta}(\pi^{k}\co)s_{\alpha_{0}}=\Ad(\pi^{\tilde{\alpha}\check{}})U_{s_{\tilde{\alpha}}\beta}(\pi^{k}\co)$
\end{center}
Comme $\tilde{\alpha}$ est la plus haute racine, si $\beta>0$ (resp. $\beta<0$), on a  $s_{\tilde{\alpha}}\beta<0$ (resp. $s_{\tilde{\alpha}}\beta>0$) et nous obtenons:
\begin{equation}
s_{\alpha_{0}}U_{\beta}(\pi^{k}\co)s_{\alpha_{0}}=\Ad(\pi^{\tilde{\alpha}\check{}})U_{s_{\tilde{\alpha}}\beta}(\pi^{k+\left\langle s_{\tilde{\alpha}}\beta ,\tilde{\alpha}\check{}\right\rangle}\co).
\label{conjal}
\end{equation}
et $\left\langle s_{\tilde{\alpha}}\beta,\tilde{\alpha}\check{}\right\rangle< 0$.
En particulier, $s_{\alpha_{0}}U_{\beta}(\pi^{k}\co)s_{\alpha_{0}}\in K_{0}$ si et seulement si 
\begin{center}
$k\geq k+\left\langle s_{\tilde{\alpha}}\beta ,\tilde{\alpha}\check{}\right\rangle\geq 1$.
\end{center}
Ainsi, nous avons $k\geq1$. De plus, si $\beta$ est en revanche négative, on a toujours $s_{\alpha_{0}}U_{\beta}(\pi\co)s_{\alpha_{0}}\in K_{1}$ en vertu de l'égalité \eqref{conjal} valable pour une racine négative, ce qui conclut.
\end{proof}

\begin{cor}
Soit $g\in\cB_{\g_{+}}$, alors $g^{-1}\g_{+}g\notin R_{\alpha_{0}}$, en particulier, il n'existe pas de droite de type $\alpha$ incluse dans $\cB_{\g_{+}}$.
\end{cor}
\begin{proof}
En effet, supposons par l'absurde qu'une telle droite existe. Posons $\g':=g^{-1}\g_{+}g$.
Alors, d'après le lemme \ref{necracine}, on a $\g'\in R_{\alpha_{0}}$, en particulier d'après la proposition \ref{inclnil}, la réduction de $\g'$ est nulle.
Or, $g\in\cB_{\g_{+}}$, donc $\g'\in M^{0}(\co)$ et est donc de réduction non nulle, une contradiction.
\end{proof}
On a la proposition suivante de même preuve que \cite[§4. Lem. 1]{KL} :
\begin{prop}\label{lemmeunb}
Soit $Y$ une composante irréductible de dimension maximale de $\cB_{\g_{+}}$. Soit $\cL$ une droite de type $\alpha$, incluse dans $\cB_{\g_{+}}$ et telle que $\cL\cap Y\neq\emptyset$, alors il existe une composante irréductible $Y'$ de dimension maximale de $\cB_{\g_{+}}$ telle que $\cL\subset Y'$.
\end{prop}
On démontre maintenant une proposition du même type que \ref{lemmedeux}:
\begin{prop}\label{un}
Soient $g, g'\in\cB_{\g_{+}}$ alors il existe des droites $\cL_{1},\dots, \cL_{n}$ de type $\alpha_{1},\dots,\alpha_{n}$ dans $\cB_{\g_{+}}$  telles que $g\in\cL_{1}$, $g'\in\cL_{n}$ et $\cL_{i}\cap\cL_{i+1}\neq\emptyset$.
En particulier, $\cB_{\g_{+}}$ est connexe.
\end{prop}
\begin{proof}
Quitte à changer $I^{\bullet}$ en $gI^{\bullet}g^{-1}$, on peut supposer que $g=1$. On écrit alors $g'=bw$ avec $b\in I$ et $w\in W_{aff}$.
Prouvons le résultat par récurrence sur la longueur de $w$, $q:=l(w)$.

Si $w=s_{\alpha}$, on pose $v=b^{-1}\g_{+}b$. Nous avons alors $v\in R_{\alpha}$ et également $\g_{+}$. En particulier, $\g_{+}$ est dans tous les semi-groupes d'Iwahori de $\hat{P}_{\alpha}^{\bullet}$, on obtient alors une droite de type $\alpha$ qui relie 1 et $g'$.

Pour passer de $q-1$ à $q$, on considère une décomposition réduite $w=s_{1}\dots s_{q}$. Soit $\alpha$ la racine correspondant à $s_{q}$, alors $w\alpha<0$.
A nouveau, en considérant $v:=b^{-1}\g_{+}b$, nous avons que $v\in R_{w}=R\cap wRw^{-1}$ qui est le radical pro-nilpotent de $\hat{P}_{\alpha}^{\bullet}$ et de même $u\in R_{w}$. A nouveau, $\g_{+}$ est dans tous les semi-groupes d'Iwahori de $\hat{P}_{\alpha}^{\bullet}$, en particulier, $\g_{+}\in s_{q}I^{\bullet}s_{q}$ et ce semi-groupe d'Iwahori est relié à $I^{\bullet}$ par une droite de type $\alpha$.
Maintenant, par hypothèse de récurrence, $s_{q}I^{\bullet}s_{q}$ est relié par une chaîne de $\bP^{1}$ à $g'I^{\bullet}(g')^{-1}$, ce qui conclut.
\end{proof}

\begin{thm}
La variété $\cB_{\g_{+}}$ est équidimensionnelle.
\end{thm}
\begin{proof}
La preuve est la même que \cite[§4. Prop. 1]{KL},  que l'on rappelle pour mémoire.
Soit $Y$ une composante irréductible  de $\cB_{\g_{+}}$ et $g\in Y$ qui n'est dans aucune autre composante irréductible. Soit $g'$ dans une composante irréductible  $Y_{0}$ de dimension maximale $d$. On peut alors d'après la proposition \ref{un} relier dans $\cB_{\g_{+}}$, $g$ et $g'$ par des droites $\cL_{1},\dots,\cL_{n}$, incluses dans $\cB_{\g_{+}}$, telles que $\cL_{i}\cap\cL_{i+1}\neq\emptyset$ et $g\in \cL_{1}$ et $g'\in\cL_{n}$.

En vertu de la proposition \ref{lemmeunb} appliqué à $(Y_{0},\cL_{n})$, on peut trouver une composante irréductible $Y_{1}$ de dimension maximale $d$ telle que $\cL_{n}\in Y_{1}$, en itérant le procédé, on trouve une suite de composantes irréductibles $(Y_{0},\dots, Y_{n})$ de dimension maximale telles que $\cL_{n-i}\in Y_{i}$. En particulier, comme $g\in\cL_{1}$, on obtient que $g\in Y_{n}$, donc $Y=Y_{n}$ et on obtient l'équidimensionnalité de $\cB_{\g_{+}}$.
\end{proof}
\subsection{Le cas où $\g$ est déployé}\label{split}
Supposons que $\gamma$ est déployé dans $T(F)$. Il s'écrit alors $\gamma=\gamma_{0}\pi^{\nu}$ où $\gamma_{0}\in T(\mathcal{O})$ et $\nu\in X_{*}(T)$, que l'on peut supposer dominant, quitte à conjuguer.
On veut  montrer que: 
\begin{prop}\label{splitdim}
Dans le cas où $\g$ est déployé, on a la formule de dimension:
\begin{center}
$\dime X_{\gamma}^{\lambda}=\left\langle \rho,\lambda\right\rangle+\frac{1}{2} d(\gamma)$.
\end{center}
\end{prop}
Le $k$-groupe $U(F)$ agit sur la grassmannienne affine $\Gr$ et les orbites sont indexées par $\nu\in X_{*}(T)$, que l'on note $S_{\nu}$. On a $S_{0}=U(F)/U(\mathcal{O})$.
Nous avons le lemme tiré de \cite[2.5]{GHKR} :
\begin{lem}
Si $Y$ un sous ensemble localement fermé de $X$ et stable par $T(F)$, alors: 
\begin{center}
$\dim(Y)=\dim(Y\cap S_{\nu})$, $\nu\in X_{*}(T)$.
\end{center}
\end{lem}
On est donc ramené à étudier la dimension de:
\begin{center}
$Y_{\gamma,\lambda}:=\{g\in U(F)/U(\mathcal{O})\vert~g^{-1}\gamma g\in K\pi^{\la} K\}$.
\end{center}
En particulier, si nous notons $f_{\gamma}:U(F)\rightarrow U(F)$ défini par:
\begin{center}
$f_{\gamma}(n)=n^{-1}\gamma n\gamma^{-1}$,
\end{center}
nous avons que $Y_{\gamma,\lambda}=f_{\gamma}^{-1}(K\pi^{\lambda}K \pi^{-\nu}\cap U(F))/U(\mathcal{O})$.
\begin{defi}
On dit qu'un sous-ensemble $Y$ de $U(F)$ est admissible s'il existe $m, n\in\mathbb{N}$ tel que $Y\subset U(\pi^{-m}\co)$ et qu'il est l'image réciproque  d'une certaine sous-variété de $U(\pi^{-m}\co)/U(\pi^{n}\co)$.
Pour un ensemble admissible $Y$ de $U(F)$, on choisit $n\geq 0$ tel que $Y$ est stable par multiplication par $U(\pi^{n}\co)$ et on pose:
\begin{center}
$\dim Y:=\dim (Y/U(\pi^{n}\co))-\dim (U(\co)/U(\pi^{n}\co))$,
\end{center}
qui est bien indépendant de $n$.
\end{defi}
Calculons la dimension de $K\pi^{\lambda}K \pi^{-\nu}\cap U(F)$ qui se déduit des calculs de Mirkovic-Vilonen \cite{MV}:

\begin{prop}\cite[Prop. 2.14.2]{GHKR} \label{ghkr}
Nous avons la formule de dimension suivante, au sens des ensembles admissibles :
\begin{center}
$\dime K\pi^{\lambda}K \pi^{-\nu}\cap U(F)=\left\langle \rho,\lambda-\nu\right\rangle$.
\end{center}
\end{prop}
Passons donc  à la preuve de la proposition \ref{splitdim}.
\begin{proof}
En identifiant $U(F)$ avec $F^{r}$, par le biais des groupes radiciels $U_{\alpha}$, le morphisme $f_{\gamma}$ agit sur chacun de ces facteurs par:
\begin{center}
$f_{\gamma}(e_{\alpha})=(\alpha(\gamma)-1)e_{\alpha}$.
\end{center}
On en déduit donc d'après la proposition $\ref{ghkr}$ que la dimension de
$f_{\gamma}^{-1}(K\pi^{\lambda}K \pi^{-\nu}\cap U(F))$ vaut:
\begin{center}
$\sum\limits_{\alpha>0}\val(1-\alpha(\gamma))+\left\langle \rho,\lambda-\nu\right\rangle$.
\end{center}
Maintenant, comme $\gamma=\gamma_{0}\pi^{\nu}\in T(F)$, nous avons l'égalité:
\begin{center}
$d(\g)=-\left\langle 2\rho,\nu\right\rangle +2\sum\limits_{\alpha>0}\val(1-\alpha(\g))$.
\end{center}
Ainsi, dans le cas déployé, la dimension de la fibre de Springer $X_{\gamma}^{\lambda}$ est:
\begin{center}
$\left\langle \rho,\lambda\right\rangle+ \frac{1}{2}\val(\det_{F}(\Id-\ad(\gamma):\mathfrak{g}(F)/\mathfrak{g}_{\gamma}(F)\rightarrow\mathfrak{g}(F)/\mathfrak{g}_{\gamma}(F)))$.
\end{center}
\end{proof}
\subsection{Fin de la preuve du théorème $\ref{dim}$}\label{finrefdim}
On a vu que l'ouvert régulier $X_{\gamma}^{\lambda,reg}$ est non vide et est une orbite sous le centralisateur de $\gamma$.
\begin{prop}\label{dimspring}
On a $\dim X_{\gamma}^{\lambda,reg}= \dime X_{\gamma}^{\lambda}=\dime\mathcal{B}_{\gamma_{+}}$.
De plus, la dimension du complémentaire de l'ouvert régulier est de dimension strictement plus petite que celle de $\dime X_{\gamma}^{\lambda}$.
\end{prop}
\begin{proof}
On a une flèche projective lisse $p:\cB\rightarrow\Gr$ et comme nous avons que $\g_{+}$ est topologiquement quasi-unipotent, elle induit une flèche surjective $p:\mathcal{B}_{\gamma_{+}}\rightarrow X_{\gamma_{+}}=X_{\g}^{\la}$. Au-dessus de $X_{\gamma}^{\lambda,reg}$, en vertu de la proposition \ref{stnil} elle est finie et en dehors du lieu régulier, les fibres sont de dimension au moins un d'après la proposition \ref{stnilb}. Comme $\mathcal{B}_{\gamma_{+}}$ est équidimensionnelle, on en déduit que 
\begin{center}
$\dime X_{\gamma}^{\lambda,reg}=\dime\mathcal{B}_{\gamma_{+}}$,
\end{center}
ainsi que l'assertion sur la dimension du complémentaire.
De plus, $\dime X_{\gamma}^{\lambda}\leq \dime\mathcal{B}_{\gamma_{+}}$, donc nous avons l'égalité.
\end{proof}
Soit $\tilde{F}$ l'extension qui déploie $\gamma$ de degré $n$. On sait que $\gamma$ est stablement conjugué à un élément de $T(\tilde{F})$.
On note: 
\begin{center}
$X_{\gamma}^{\lambda}:=\{g\in G(F)/G(\mathcal{O})\vert~g^{-1}\gamma g\in K\pi^{\la}K\}$
\end{center}
(resp $\tilde{X}_{\gamma}^{\lambda}$,  la même chose dans $\tilde{F})$.
Le centralisateur de $\gamma$ agit sur $X_{\gamma}^{\lambda}$ et $\tilde{X}_{\gamma}^{\lambda}$.
Si $g$ est un point de $X_{\gamma}^{\lambda}$ (resp. $\tilde{X}_{\gamma}^{\lambda}$), on note $\co_{g}$ (resp. $\tilde{\co}_{g}$) l'orbite pour $G_{\gamma}(F)$ (resp. $G_{\gamma}(\tilde{F})$).
Nous avons déjà vu que $X_{\g}^{\la,reg}$ est une orbite sous le centralisateur $G_{\gamma}$ .
On rappelle alors la formule de Bezrukavnikov, dont la preuve est rigoureusement la même dans le cas qui nous concerne:
\begin{prop}[\cite{B}, Lem. 2-3]\label{dimspring2} 
On a l'égalité suivante:
\begin{center}
$\dime\tilde{\mathcal{O}}_{g}=n[\dime\mathcal{O}_{g}+\frac{1}{2}\defa(\gamma)]$.
\end{center}
\end{prop}
On peut  terminer la preuve du théorème \ref{dim}:
\begin{proof}
On sait, d'après l'étude du cas déployé, que $\tilde{X}_{\gamma}^{\lambda}$ est de dimension: 
\begin{center}
$n[\left\langle\rho,\lambda \right\rangle +\frac{1}{2}\delta(\gamma)]$.
\end{center}
En combinant les propositions $\ref{dimspring}$ et $\ref{dimspring2}$ ainsi que le calcul dans le cas déployé, on a la formule désirée pour la dimension de la fibre de Springer:
\begin{center}
$\dime X^{\lambda}_{\gamma}=\left\langle \rho,\lambda\right\rangle+\frac{1}{2}[\delta(\gamma)-\defa(\gamma)]$.
\end{center}
\end{proof}

\begin{flushleft}
Alexis Bouthier \\
Université Paris-Sud UMR 8628\\
Mathématiques, Bâtiment 425, \\
F-91405 Orsay Cedex France \\
E-mail: alexis.bouthier@math.u-psud.fr \\
\end{flushleft}
\end{document}